\newif\ifjournal
\journaltrue

\ifjournal

	\documentclass[a4paper]{article}

	\usepackage[english]{babel}
	\usepackage[utf8]{inputenc}
	\usepackage[T1]{fontenc}
	\usepackage{csquotes}

	\usepackage[a4paper,%
		top=3cm, bottom=3cm ,left=3cm, right=2.9cm,%
		marginparwidth=2cm, footskip=2.0cm]%
		{geometry}

	\usepackage[
		style=numeric-comp,
		hyperref=true,
		backend=biber,
		doi=false,
		url=false,
		isbn=false,
		backref=false,
		giveninits=true,
		sorting=nyt,
		maxcitenames=3,
		maxbibnames=100,
		block=none]{biblatex}
	\usepackage[colorlinks=true, allcolors=blue]{hyperref}

	\DeclareNameAlias{default}{family-given}
	\newbibmacro*{string+doiurlisbn}[1]{
		\iffieldundef{doi}{
			\iffieldundef{url}{
				\iffieldundef{isbn}{
					\iffieldundef{issn}{#1}
						{\href{http://books.google.com/books?vid=ISSN\thefield{issn}}{#1}}
				}
						{\href{http://books.google.com/books?vid=ISBN\thefield{isbn}}{#1}}
			}
			{\href{\thefield{url}}{#1}}
		}
		{\href{http://dx.doi.org/\thefield{doi}}{#1}}
	}
	\DeclareFieldFormat*{title}{\usebibmacro{string+doiurlisbn}{#1\isdot}}
	\renewcommand*{\newunitpunct}{\addcomma\space}
	\DeclareFieldFormat{journaltitle}{\mkbibemph{#1}\newunitpunct}
	\renewbibmacro{in:}{}
	\AtEveryBibitem{\clearfield{number}}
	\DeclareSourcemap{
		\maps[datatype=bibtex]{
			\map[overwrite=true]{
				\step[fieldsource=fjournal]
				\step[fieldset=journal, origfieldval]
			}
		}
	}
	\title{Positivity preservation of implicit discretizations of the advection equation}
	\author{Yiannis Hadjimichael\thanks{{\texttt{hadjimichael@wias-berlin.de}}, MTA-ELTE Numerical
	Analysis and Large Networks Research Group, E\"otv\"os Lor\'and University, P\'azm\'any P\'eter
	s\'et\'any 1/C, H-1117, Budapest, Hungary, and Weierstrass Institute (WIAS), Mohrenstra{\ss}e 39,
	10117, Berlin, Germany.}
	\and David I. Ketcheson\thanks{{\texttt{david.ketcheson@kaust.edu.sa}}, King Abdullah University of
	Science and Technology (KAUST), Computer Electrical and Mathematical Science and Engineering Division
	(CEMSE), Thuwal, 23955-6900, Saudi Arabia.}
	\and Lajos L\'oczi\thanks{{\texttt{LLoczi@inf.elte.hu}}, Department of Numerical Analysis, E\"otv\"os
	Lor\'and University, and Department of Differential Equations, Budapest University of Technology and
	Economics, Hungary.}}

\else

	\documentclass[a4paper,12pt,twoside,english]{article}
	\usepackage{wiaspreprint}

	\pagefootleft{WIAS Preprint No. +++preprint number+++}

	\author[Y. Hadjimichael, D. I. Ketcheson, L. L\'oczi]{%
	Yiannis Hadjimichael\footnote{Weierstrass Institute, \\ Mohrenstr. 39, \\ 10117 Berlin, \\ Germany \\
	E-Mail: yiannis.hadjimichael@wias-berlin.de},
	David I. Ketcheson\footnote{ Computer, Electrical and Mathematical Science \\ and Engineering Division
	(CEMSE), \\ King Abdullah University of Science \\ and Technology (KAUST), \\ Thuwal  23955-6900, \\
	Saudi Arabia \\
	E-Mail: david.ketcheson@kaust.edu.sa},
	Lajos L\'oczi\footnote{Department of Numerical Analysis \\ E\"otv\"os Lor\'and University, and \\
	Department of Differential Equations \\
	Budapest University of Technology and Economics \\ Hungary \\
	E-Mail: LLoczi@inf.elte.hu}
	}
	\title[Positivity preservation of implicit discretizations of the advection equation]
	{Positivity preservation of implicit discretizations of the advection equation}
	\nopreprint{+++preprint number+++}	
	\nopreyear{2021}	
	\thanks{The Application-domain Specific Highly Reliable IT Solutions project  has been implemented
	 with the support provided from the National Research, Development and Innovation Fund of Hungary,
	 financed under the Thematic Excellence Programme TKP2020-NKA-06 (National Challenges Subprogramme)
	 funding scheme.
	 This work was supported by the King Abdullah University of Science and Technology (KAUST), 4700
	 Thuwal, 23955-6900, Saudi Arabia, and by the Leibniz Competition}

	\makeatletter
	\def\@textbottom{\vskip \z@ \@plus 32pt}
	\let\@texttop\relax

\fi

\usepackage{amsmath,amsthm,amssymb}
\usepackage{mathtools}
\usepackage{graphicx}
\usepackage{cleveref}
\usepackage{enumitem}
\usepackage[colorinlistoftodos]{todonotes}

\graphicspath{{figures/}}

\newtheorem{theorem}{Theorem}
\newtheorem{lemma}{Lemma}
\newtheorem{remark}{Remark}
\newtheorem{example}{Example}
\newtheorem{corollary}{Corollary}
\newtheorem{proposition}{Proposition}

\newcommand{\dt}{\Delta t}
\newcommand{\dx}{\Delta x}
\newcommand{\te}{\theta}
\newcommand{\nul}{\nu_L(k,\theta)}
\newcommand{\nur}{\nu_R(k,\theta)}
\newcommand{\yl}{y_L(k,\theta)}
\newcommand{\yr}{y_R(k)}
\newcommand{\nplus}{\mathbb{N}^+}
\newcommand{\rr}{\mathbb{R}}
\newcommand{\Por}{P_{R,k}(y)}
\newcommand{\Pol}{P_{L,k,\te}(y)}
\newcommand{\cP}{{\cal{P}}}
\newcommand{\cD}{{\cal{D}}}
\newcommand{\cF}{{\cal{F}}}

\newcommand{\Mod}[1]{\ (\mathrm{mod}\ #1)}
\newcommand\blfootnote[1]{%
  \begingroup
  \renewcommand\thefootnote{}\footnote{#1}%
  \addtocounter{footnote}{-1}%
  \endgroup
}

\begin{document}

\maketitle

\begin{abstract}
We analyze, from the viewpoint of positivity
preservation, certain discretizations of a fundamental partial differential
equation, the one-dimensional advection equation with periodic boundary
condition. The full discretization is obtained by coupling a finite difference spatial semi-discretization
(the second- and some higher-order centered difference schemes, or the Fourier spectral collocation method)
with an arbitrary $\te$-method in time
(including the forward and backward Euler methods, and a second-order method by
choosing $\te\in [0,1]$ suitably). 
The full discretization generates a
two-parameter family of circulant matrices $M\in\mathbb{R}^{m\times m}$, 
where each matrix entry is a
rational function in $\te$ and $\nu$. Here, $\nu$ denotes the CFL number,
being proportional to the ratio between the temporal and spatial discretization
step sizes. The entrywise non-negativity of the matrix $M$---which is
equivalent to the positivity preservation of the fully discrete
scheme---is investigated via discrete Fourier analysis and also by solving some low-order parametric 
linear recursions. We find that positivity preservation of the fully discrete system is impossible if the number  of spatial grid points $m$ is even. However, it turns out that positivity preservation of the fully discrete system is recovered for \emph{odd} values of $m$
provided that $\te\ge 1/2$ and $\nu$ are chosen suitably. 
These results are interesting since the systems of ordinary differential equations obtained via the spatial semi-discretizations studied are \emph{not} positivity preserving. 
\end{abstract}
\blfootnote{The Application-domain Specific Highly Reliable IT Solutions project  has been implemented with the support provided from the National Research, Development and Innovation Fund of Hungary, financed under the Thematic Excellence Programme TKP2020-NKA-06 (National Challenges Subprogramme) funding scheme.
This work was supported by the King Abdullah University of Science and Technology (KAUST), 4700
Thuwal, 23955-6900, Saudi Arabia, and by the Leibniz Competition.}

\section{Background and motivation}
In this work, we investigate the positivity of some discretizations of the advection equation
with periodic boundary condition
\begin{align}\label{advection}
\begin{split}
	U_t (x,t) &= a U_x(x,t), \qquad x\in[0,1], t>0, \\
	U(x,0) &= U_0(x), \\
	U(0,t) &= U(1,t),
\end{split}
\end{align}
where $U:\mathbb{R}\times [0,+\infty)\to\mathbb{R}$ is the unknown function, $U_0:\mathbb{R}\to\mathbb{R}$ is a given differentiable initial function, and $a>0$ is a constant. 
The exact solution of the Cauchy problem \eqref{advection}, given by $U(x,t) =
U_0(\{x+a t\})$ (where $\{\cdot\}$ denotes the fractional part), is positivity preserving; i.e.
\begin{align} \label{implies-positivity}
\forall x\in[0,1], \forall t>0\quad\quad U_0(x) \ge 0 \implies U(x,t) \ge 0.
\end{align}
Positivity is often important in this context, since $U$ may represent a
concentration or density that cannot be negative.

\begin{remark}
Herein the term {\emph{positivity}} is always meant in the weak sense;
i.e.~it means {\emph{non-negativity}}. 
\end{remark}

Finite difference spatial semi-discretization of \eqref{advection} on a uniform grid $\{\dx,2\dx,\ldots, m\dx\}\subset [0,1]$ 
with mesh spacing $\dx>0$ and $m\dx=1$ yields a system of ordinary differential equations
\begin{align} \label{semi-discrete}
    u'(t) & = \frac{a}{\dx}Lu(t),
\end{align}
where $u:\mathbb{R}\to\mathbb{R}^m$, $L\in\mathbb{R}^{m\times m}$ is a circulant matrix  \cite[Section 5.16]{matmat}, and $m \in\mathbb{N}^+$ is the number of grid points (the points $x=0$ and $x=1$ are identified due to the periodic boundary condition).
If one uses an upwind spatial discretization

\begin{equation}\label{upwind1}
L=\left(
\begin{array}{cccccc}
 -1 & 1 & 0  & \cdots & 0 & 0 \\
 0 & -1 & 1 &  \cdots & 0 & 0 \\
 0 & 0 & -1 &  \ddots & 0 & 0\\
 \vdots  & \vdots  & \ddots  & \ddots & \ddots & \vdots \\
  0 & 0 & \cdots  & 0 & -1 & 1 \\
 1 & 0 & \cdots & 0 & 0 & -1 \\
\end{array}
\right),
\end{equation}
then the exact solution of \eqref{semi-discrete} is also positivity
preserving. Moreover, a corresponding full discretization will be positivity preserving too,
under an appropriate time step size restriction $0<\dt\le\dt_0$
if, for example, the forward (explicit) Euler method or any strong stability preserving
method \cite{SSPbook} is used in time; see, e.g.~\cite{posconv}.  

Positivity-preserving methods for transport equations are typically based
on low-order upwind-biased spatial discretizations like that above, or involve
nonlinear limiters (or both).
Here we instead consider the positivity of linear higher-order centered discretizations.
A second-order scheme is obtained with the centered difference discretization
\begin{equation}\label{centered-difference}
L=\left(
\begin{array}{cccccc}
 0 & \frac{1}{2} & 0  & \cdots & 0 & -\frac{1}{2} \\
 -\frac{1}{2} & 0 & \frac{1}{2} &  \cdots & 0 & 0 \\
 0 & -\frac{1}{2} & 0 &  \ddots & 0 & 0\\
 \vdots  & \vdots  & \ddots  & \ddots & \ddots & \vdots \\
  0 & 0 & \cdots  & -\frac{1}{2} & 0 & \frac{1}{2} \\
 \frac{1}{2} & 0 & \cdots & 0 & -\frac{1}{2} & 0 \\
\end{array}
\right).
\end{equation}
However, this spatial semi-discretization is not positivity preserving, since the
matrix $L$ has at least one negative off-diagonal entry \cite[Chapter I,
Theorem 7.2]{hundsdorferverwer}.
This implies that any consistent full discretization based on \eqref{centered-difference}
must fail to preserve positivity under sufficiently small step sizes $\dt>0$.
Indeed, a full discretization based on the scheme \eqref{centered-difference}
and forward Euler in time is not positivity preserving for any step size $\dt>0$.

On the other hand, interestingly, using \eqref{centered-difference} with backward (implicit) 
Euler time integration, one 
observes positivity preservation
under \emph{large} enough time step sizes provided that the parity of the number of spatial grid points is \emph{odd}.  To investigate the differences between the behavior
of the forward and backward Euler methods, we will study the $\theta$-method 
\cite[Chapter IV.3]{hairerwanner} as time discretization applied to \eqref{semi-discrete}
\begin{align}\label{firsttheta}
    u^{n+1} = u^n + \frac{a\dt}{\dx}((1-\theta)Lu^n + \theta Lu^{n+1}).
\end{align}
For $\te\in[0,1]$, the $\theta$-family includes both Euler methods as limiting cases: the forward Euler method for $\te=0$, the backward 
 Euler method for $\te=1$, and the only second-order $\te$-method for $\te=1/2$; 
 any $\theta$-method with $\theta\in(0,1]$ is implicit.
 \begin{remark}
As is customary in the context of space-time discretizations of partial differential equations, superscripts of $u$ in \eqref{firsttheta} (and later in this work) are not exponents but denote time discretization steps.
\end{remark}
\begin{remark}
Similarly to \eqref{implies-positivity}, the semi-discrete system \eqref{semi-discrete} and the fully discrete system \eqref{firsttheta} are said to be \emph{positivity preserving}, if for \textbf{any} componentwise non-negative vector of initial condition 
\begin{itemize}
\item $u(0)$, the solution $u(t)$ of \eqref{semi-discrete} stays componentwise non-negative $\forall t>0$
\item $u^0$, the solution $u^n$ of \eqref{firsttheta} stays componentwise non-negative $\forall n\in\mathbb{N}^+$,
\end{itemize}
respectively.
\end{remark}

The motivation for this work is not to develop new positivity preserving
methods, but to study the positivity of some of fundamental discretizations
such as the second-order centered difference method \eqref{centered-difference} and the
$\theta$-method \eqref{firsttheta}.  As we will see, the combination of
these methods does not preserve positivity in general, nor in the limit
of small time step size, so it is not typically recommended in practice.
Nevertheless, this study may both shed light on the
behavior of more complicated methods used in practice and provide tools
that can be used to study the positivity of those methods.
In the later sections of the paper, we combine higher-order methods in space with the $\theta$-method in time, as a next step in this direction.


\subsection{Structure of the paper and notation}
In Section~\ref{sectiondiscFourier}, we first characterize
positivity preservation of full discretizations of \eqref{advection} resulting from 
finite difference spatial and one-step time discretizations. Then, in Section~\ref{sectioncentered}, we study positivity of the second-order centered differences in space
combined with the $\theta$-method in
time, using discrete Fourier
analysis. We also point out some connections with structured non-negative inverse eigenvalue problems. In Section~\ref{section3}, we study this particular full discretization in more detail.
In Section \ref{explsect}, by setting up and solving certain parametric linear recursions, we derive explicit, non-trigonometric formulae for the entries of the full discretization matrix $M$. Then, in Section \ref{nonnegsect}, we use these formulae to provide precise results on the non-negativity of $M$ in terms of roots of some  sparse polynomials.
In Section \ref{otherdisc}, we discuss some observations and results regarding
higher-order spatial discretizations, including high-order centered differences
(in Section~\ref{highercentered}) and spectral collocation methods (in Section~\ref{spectrcoll}).
We summarize our findings in Section \ref{conclusions}.\\

Throughout the paper, the set of positive integers is denoted by $\nplus$, the complex imaginary unit is $\imath$, the identity matrix is $I\in\rr^{m\times m}$, and to emphasize the dimensions of a matrix, we will sometimes write, for example, $L_{m\times m}$.
The symbol $M\ge 0$ means that  $M_{i,j}\ge 0$ for every entry $1\le i, j\le m$ of the matrix $M\in\mathbb{R}^{m\times m}$. The positive integer $m$ is the number of spatial grid points within the interval $[0,1]$, and the matrices $L$ and $M$ are the matrices corresponding to the spatial and the full discretizations, respectively. The three key parameters in our investigations will be $m$, $\theta\in [0,1]$ and $\nu>0$ (see \eqref{firsttheta} and \eqref{nudef}).

The computations in this work have been carried out by using Wolfram \textit{Mathematica} version 11.





\section{Discrete Fourier analysis}\label{sectiondiscFourier}
From here on, we consider the problem \eqref{advection} on the domain $x\in[0,1]$
with periodic boundary condition $U(0,t)=U(1,t)$.  Finite difference discretization
in space with step size $\dx$ leads to \eqref{semi-discrete} with $u_j\approx u(j\dx)$ for $1\le j\le m$. 
The circulant matrix $L\in\mathbb{R}^{m\times m}$ has the eigendecomposition
\begin{align}\label{eigen}
    L  = \cF \Lambda \cF^*,
\end{align}
where the (unitary) matrix of normalized eigenvectors $\cF$ has entries
\begin{align}\label{eigenvectrors}
    f_{j,\ell}  \coloneqq \frac{1}{\sqrt{m}}\exp(\imath  (j-1) \xi_\ell)  \quad\quad (1 \le j, \ell \le m),
\end{align}
and $\Lambda$ is the diagonal matrix of eigenvalues $\lambda_\ell$, which depends on the
particular finite difference method chosen.  Here $\xi_\ell$ are evenly spaced angles 
\begin{equation}\label{xildef}
    \xi_\ell  \coloneqq \frac{2\pi(\ell-1)}{m} \quad\quad (1 \le \ell \le m),
\end{equation}
such that $\exp(\imath\xi_\ell)$ are the $m^\text{th}$ roots of unity.
Applying a one-step time
discretization with step size $\dt$ and stability function $R:\mathbb{C}\to\mathbb{C}$ to \eqref{semi-discrete} leads to
the iteration
\begin{align}\label{M}
    u^{n+1} & = M u^n,
\end{align}
where 
\begin{equation}\label{Mdualdef}
M:=R(\nu L)=\cF R(\nu \Lambda) \cF^*, 
\end{equation}
and 
\begin{equation}\label{nudef}
\nu:=a\frac{\dt}{\dx}>0
\end{equation}
is the CFL number.  
Then it is easily seen that 
\[
\boxed{\text{positivity preservation of the fully discrete numerical solution}\quad \Longleftrightarrow \quad M\ge 0.}
\]
For one-step methods, $R$ is a rational function, and products and inverses of circulant matrices are also circulant \cite[Fact 5.16.7]{matmat}, so
$M$ is also a real, circulant matrix.
Thus it is defined completely by the entries of its first row, which
are given by
\begin{align} \label{M-entries}
    M_{1,j} & = \frac{1}{m} \sum_{\ell=1}^m R(\nu\lambda_\ell) \exp(-\imath(j-1)\xi_\ell) \quad\quad (1\le j\le m).
\end{align}

\subsection{Second-order centered discretization in space, \texorpdfstring{$\theta$}{}-method in time}\label{sectioncentered}
In what follows we assume $m \ge 3$. Consider the case of a 3-point centered difference approximation in space (having order 2):
\[
    U_x\Big|_{x=x_j} \approx \frac{u_{j+1}-u_{j-1}}{2\dx},
\]
so that $L\in\mathbb{R}^{m\times m}$ is a circulant matrix with entries $(-1/2, 0, 1/2)$ on the central
three diagonals:
\begin{equation}\label{Ldef}
L:=\left(
\begin{array}{cccccc}
 0 & \frac{1}{2} & 0  & \cdots & 0 & -\frac{1}{2} \\
 -\frac{1}{2} & 0 & \frac{1}{2} &  \cdots & 0 & 0 \\
 0 & -\frac{1}{2} & 0 &  \ddots & 0 & 0\\
 \vdots  & \vdots  & \ddots  & \ddots & \ddots & \vdots \\
  0 & 0 & \cdots  & -\frac{1}{2} & 0 & \frac{1}{2} \\
 \frac{1}{2} & 0 & \cdots & 0 & -\frac{1}{2} & 0 \\
\end{array}
\right),
\end{equation}
that is,
\begin{subequations}
\begin{align}
    L_{i,i-1} & \coloneqq -\frac{1}{2},\quad\quad L_{i,i+1} \coloneqq  \frac{1}{2} \\
    L_{1,m} & \coloneqq -\frac{1}{2}, \quad \quad L_{m,1} \coloneqq \frac{1}{2}.
\end{align}
\end{subequations}
It is known that the eigenvalues of $L$ are 
\begin{equation}\label{Leigenvalues}
\lambda_\ell=\imath \sin(\xi_\ell)\quad (1\le\ell\le m).
\end{equation}
Now we consider the $\theta$-method \cite[Chapter IV.3]{hairerwanner} in time, whose stability function 
is
\begin{align}\label{R}
    R(z)  \coloneqq \frac{1+(1-\theta)z}{1-\theta z},
\end{align}
so with the second-order centered difference in space we get from \eqref{M-entries} for $1\le j\le m$ that
the entries of the full discretization matrix are 
\begin{equation}\label{matrix_entries}
\begin{split}
	M_{1,j} &= \frac{1}{m} \sum_{\ell=1}^m \frac{1+(1-\theta)\nu \imath\sin(\xi_\ell)}
		{1-\theta\nu \imath\sin(\xi_\ell)}\exp\left(-\imath(j-1)\xi_\ell\right) \\
		&= \frac{1}{m} \sum_{\ell=1}^{m} \frac{\left(1-\theta(1-\theta)\nu^2\sin^2(\xi_\ell)\right)
		\cos((j-1)\xi_\ell) + \nu \sin(\xi_\ell)\sin((j-1)\xi_\ell)}{1+\theta^2\nu^2 \sin^2(\xi_\ell)} \\
		&\quad - \frac{\imath}{m} \sum_{\ell=1}^{m} \frac{\left(1-\theta(1-\theta)\nu^2\sin^2(\xi_\ell)
		\right)\sin((j-1)\xi_\ell) - \nu \sin(\xi_\ell)\cos((j-1)\xi_\ell)}
		{1+\theta^2\nu^2 \sin^2(\xi_\ell)}.
\end{split}
\end{equation}
Note that the angles $\{(j-1)\xi_\ell\}_{\ell=1}^m$ are symmetric about the $x$-axis if $m$ is odd, and
also if $m$ is even and $j$ is odd.
If both $m$ and $j$ are even, then the angles are symmetric about the origin.
Therefore, we have that
\begin{align*}
	\sum_{\ell=1}^{m} \sin((j-1)\xi_\ell) = 0 \qquad \text{and} \qquad
		\sum_{\ell=1}^{m} \sin(\xi_\ell)\cos((j-1)\xi_\ell) = 0,
\end{align*}
for all $1\le j\le m$  and for any value of $m$.
Moreover the factors
$\frac{1-\theta(1-\theta)\nu^2\sin^2(\xi_\ell)}{1+\theta^2\nu^2 \sin^2(\xi_\ell)}$
and $\frac{\nu}{1+\theta^2\nu^2 \sin^2(\xi_\ell)}$ in \eqref{matrix_entries} keep this symmetry.
Thus, the imaginary part of \eqref{matrix_entries} vanishes, yielding $M_{1,j}\in\mathbb{R}$ for all $j$,
as expected.
So for $1\le j\le m$ we get
\[
  M_{1,j}  = \frac{1}{m} \sum_{\ell=1}^{m} \frac{\left(1-\theta(1-\theta)\nu^2\sin^2(\xi_\ell)\right)
  \cos((j-1)\xi_\ell) + \nu \sin(\xi_\ell)\sin((j-1)\xi_\ell)}{1+\theta^2\nu^2 \sin^2(\xi_\ell)}.
\]
We will also make use of the identities
\begin{align} \label{angle-identities}
    \cos((m-1)\xi_\ell) & = \cos(\xi_\ell), & \sin((m-1)\xi_\ell) & = -\sin(\xi_\ell).
\end{align}
This leads to the following expressions for the first, second and last entries of the first
row of $M$:
\begin{align*}
	M_{1,1} & = \frac{1}{m} \sum_{\ell=1}^m \frac{1-\theta(1-\theta)\nu^2\sin^2(\xi_\ell)}
		{1+\theta^2\nu^2\sin^2(\xi_\ell)}, \\
	M_{1,2} & = \frac{1}{m} \sum_{\ell=1}^{m} \frac{\left(1-\theta(1-\theta)\nu^2\sin^2(\xi_\ell)\right)
		\cos(\xi_\ell) + \nu \sin^2(\xi_\ell)}{1+\theta^2\nu^2 \sin^2(\xi_\ell)}, \\
	M_{1,m} & = \frac{1}{m} \sum_{\ell=1}^m \frac{\left(1-\theta(1-\theta)\nu^2\sin^2(\xi_\ell)\right)
		\cos(\xi_\ell) - \nu \sin^2 (\xi_\ell)}{1+\theta^2\nu^2 \sin^2 (\xi_\ell)}.
\end{align*}
These entries will have a special role in the forthcoming analysis. 
We distinguish two cases.
\begin{description}[style=unboxed,leftmargin=0cm]
\item [{Case 1:} $m$ is {even}.]
\item \noindent By using similar symmetry arguments as before, we conclude that for {even} $m=2k\ge 4$
the entries of matrix $M$ are given by
\begin{align*}
	M_{1,j} = \begin{cases}
				\displaystyle
				\frac{1}{m} \sum_{\ell=1}^{m}\frac{\left(1-\theta(1-\theta)\nu^2\sin^2(\xi_\ell)
					\right)\cos((j-1)\xi_\ell)}{1+\theta^2\nu^2\sin^2(\xi_\ell)}, &\mbox{if } j
					\text{ is odd}, \\[20pt]
				\displaystyle
				\frac{1}{m} \sum_{\ell=1}^{m} \frac{\nu \sin(\xi_\ell)\sin((j-1)\xi_\ell)}
					{1+\theta^2\nu^2 \sin^2 (\xi_\ell)}, &\mbox{if } j \text{ is even}.
			  \end{cases}
\end{align*}
Considering  the above expression for $j = m$ we have
\begin{align*}
    M_{1,m} & =  \frac{1}{m} \sum_{\ell=1}^{m} \frac{- \nu \sin^2 (\xi_\ell)}{1+\theta^2\nu^2 \sin^2 (\xi_\ell)} < 0.
\end{align*}
Thus the discretization using second-order centered differences in space and the $\theta$-method in time cannot preserve positivity when $m$ is even, regardless of the
values of $\theta\in[0,1]$ and $\nu>0$.
We can arrive at the same conclusion by observing that
for any $m=2k\ge 4$ we have $M_{1,2}=-M_{1,m}$, so that one of
these (non-zero) entries must always be negative.
\item [{Case 2:} $m$ is {odd}.]
\item \noindent Let us now consider the case of {odd} $m=2k+1\ge 3$. Then $\sin(\xi_\ell)\ne0$ for
$2\le\ell\le m$.
Writing
\ifjournal
\begin{align*}
	M_{1,1} & = \frac{1}{m} \left(1 + \sum_{\ell=2}^{m} \frac{1-\theta(1-\theta)\nu^2\sin^2(\xi_\ell)}
		{1+\theta^2\nu^2\sin^2(\xi_\ell)}\right) \\
	M_{1,j} & = \frac{1}{m} \left(1 + \sum_{\ell=2}^{m} \frac{(1-\theta(1-\theta)\nu^2\sin^2(\xi_\ell))
		\cos((j-1)\xi_\ell) + \nu \sin(\xi_\ell)\sin((j-1)\xi_\ell)}{1+\theta^2\nu^2 \sin^2(\xi_\ell)}
		\right) \quad (j\ge 2),
\end{align*}
\else
\begin{align*}
	M_{1,1} & = \frac{1}{m} \left(1 + \sum_{\ell=2}^{m} \frac{1-\theta(1-\theta)\nu^2\sin^2(\xi_\ell)}
		{1+\theta^2\nu^2\sin^2(\xi_\ell)}\right) \\
	M_{1,j} & = \frac{1}{m} \left(1 + \sum_{\ell=2}^{m} \frac{(1-\theta(1-\theta)\nu^2\sin^2(\xi_\ell))
		\cos((j-1)\xi_\ell) + \nu \sin(\xi_\ell)\sin((j-1)\xi_\ell)}{1+\theta^2\nu^2 \sin^2(\xi_\ell)}
		\right) \\ && \hspace*{-30pt} (j \ge 2).
\end{align*}
\fi
and taking $\nu \to +\infty$ with $m$ and $\te\in(0,1]$ fixed, we find that
\ifjournal
\begin{align*}
    M_{1,1}^\infty \coloneqq \lim_{\nu \to +\infty} M_{1,1} & = \frac{1}{m} \left(1 - \sum_{\ell=2}^m \frac{1-\theta}{\theta}\right) =  1-\frac{m-1}{m\theta} \\
    M_{1,j}^\infty \coloneqq \lim_{\nu \to +\infty} M_{1,j} & = \frac{1}{m} \left(1 - \sum_{\ell=2}^m \frac{1-\theta}{\theta}\cos((j-1)\xi_\ell)\right) = \frac{1}{m} \left(1+\frac{1-\theta}{\theta}\right) = \frac{1}{m\theta} \quad (j \ge 2).
\end{align*}
\else
\begin{align*}
    M_{1,1}^\infty \coloneqq \lim_{\nu \to +\infty} M_{1,1} & = \frac{1}{m} \left(1 - \sum_{\ell=2}^m \frac{1-\theta}{\theta}\right) =  1-\frac{m-1}{m\theta} \\
    M_{1,j}^\infty \coloneqq \lim_{\nu \to +\infty} M_{1,j} & = \frac{1}{m} \left(1 - \sum_{\ell=2}^m \frac{1-\theta}{\theta}\cos((j-1)\xi_\ell)\right) = \frac{1}{m} \left(1+\frac{1-\theta}{\theta}\right) = \frac{1}{m\theta} \\ && \hspace*{-20pt} (j \ge 2).
\end{align*}
\fi
We see that
\[M_{1,j}^\infty>0 \text{ for all } 2\le j\le m \text{ and } \te\in(0,1],
\]
 while 
\[M_{1,1}^\infty>0 \quad\Longleftrightarrow\quad \theta>\frac{m-1}{m}.
\]
Thus for fixed $m\ge 3$ and $\theta>\frac{m-1}{m}$, the matrix $M$ is non-negative 
if $\nu>0$ is large enough.

We now show that $M\ge 0$ also holds for $\theta=\frac{m-1}{m}$ with $m$ fixed and for $\nu>0$ large enough. 
Clearly, we only need to verify the non-negativity of entry $M_{1,1}$ for $\nu>0$ large enough. 
In fact, for $\theta=\frac{m-1}{m}$ and for \emph{any} $\nu>0$ we have $M_{1,1}> 0$. To see this, consider a summand with $2\le\ell\le m$ in $M_{1,1}$: 
\[
 \frac{1-\theta(1-\theta)\nu^2\sin^2(\xi_\ell)}{1+\theta^2\nu^2\sin^2(\xi_\ell)}\Bigg|_{\theta=\frac{m-1}{m}}=
 \frac{m^2-(m-1) \nu ^2 \sin ^2\left(\xi _\ell\right)}{m^2+(m-1)^2 \nu ^2 \sin ^2\left(\xi _\ell\right)}=:\varphi(\nu,\ell).
\]
Its partial derivative with respect to $\nu$ is
\[
\partial_\nu\varphi(\nu,\ell)=-\frac{2 (m-1) m^3 \nu  \sin ^2\left(\xi _\ell\right)}{\left(m^2+(m-1)^2 \nu ^2 \sin ^2\left(\xi _\ell\right)\right)^2}<0,
\]
and $\varphi(0,\ell)=1$, hence the function
\begin{align*}
	\nu\mapsto M_{1,1}\Big|_{\theta=\frac{m-1}{m}} =
		\frac{1}{m} \left(1 + \sum_{\ell=2}^{m} \varphi(\nu,\ell)\right)
\end{align*}
is positive at $\nu=0$, strictly decreases, and its limit when $\nu\to +\infty$ is
$M_{1,1}^\infty\big|_{\theta=\frac{m-1}{m}}=0$, completing the proof of the claim.\\

Summarizing the above, we have proved the following for any $\nu>0$.
\end{description}
\begin{theorem}\label{thm1}
Consider the advection equation \eqref{advection} with periodic boundary condition discretized using
$2^{\text{nd}}$-order centered differences in space and the $\theta$-method in time with $m\ge 3$ spatial grid
points and $\te\in[0,1]$.
The full discretization takes the form \eqref{M},
where\\
(i) if $m$ is even, then $M$ has at least one negative entry;\\
(ii) if $m$ is odd and $\theta\in\left[\frac{m-1}{m},1\right]$, then for large enough $\dt$ all
entries of $M$ are non-negative.
\end{theorem}

A refinement of Theorem~\ref{thm1} for odd values of $m$ will be given at the end of
Section~\ref{section3}; see Theorem~\ref{thm2}.
As for the interval $\theta\in\left[\frac{m-1}{m},1\right]$ appearing in Theorem~\ref{thm1}, see also
Figures~\ref{fig_variousk}--\ref{fig_boundary}.

\begin{remark} 
In the formulae leading to Theorem~\ref{thm1} we used a trigonometric representation of the matrix
\emph{entries} $M_{1,j}$. Here we highlight a related approach to studying the non-negativity of $M$ by
relying only on the \emph{eigenvalues} $\sigma_\ell$ ($1\le \ell\le m$) of $M$. According to
\eqref{Mdualdef}, \eqref{Leigenvalues} and \eqref{R}, we have
\[
\sigma_\ell:=R(\nu\lambda_\ell)=\frac{1+ (1-\theta )\nu \imath \sin \left(\xi _\ell\right)}{1- \theta\nu \imath  \sin \left(\xi _\ell\right)}.
\]
The main question in the context of \emph{non-negative inverse eigenvalue problems} is to 
find (necessary or sufficient) conditions for a set $\Sigma:=\{\sigma_1,\ldots,\sigma_m\}\subset\mathbb{C}$ to be the spectrum of \emph{some} non-negative $m \times m$ matrix.
One such condition is the following. It is known \cite[Chapter 4]{nonnegmatr} that if $\Sigma$ is the spectrum of 
an $m \times m$ non-negative matrix, then 
\begin{equation}\label{sigmampq}
\forall\, p, q \in\nplus:\quad\quad 0\le \left(\sum_{j=1}^m \sigma_j^{\,p}\right)^q\le m^{q-1} \sum_{j=1}^m \sigma_j^{\,p q}.
\end{equation}
For example, for $m=5$ and $\te=1$, \eqref{sigmampq} with  $p\in\{1,\ldots,9\}$ and $q\in\{2,3\}$ yields the lower bounds
\begin{equation}\label{nustartpq}
\nu\ge \nu_{*}(p,q),
\end{equation}
where the approximate values of $\nu_{*}(p,q)$ are given below:
\[
\begin{array}{|c|c|c|c|c|c|c|c|c|c|}
\hline
\nu_{*}(p,q) & p=1 & p=2 & p=3 & p=4 & p=5 & p=6 & p=7 & p=8 & p=9 \\
\hline
q=2 & 3.0074 & 1.462 & 0.9669 & 0.7219 & 0.5753 & 0.4778 & 0.4082 & 0.3563 & 0.3160 \\
\hline
q=3 & 2.1497 & 1.0269 & 0.6694  & 0.4941 & 0.3907 & 0.3227  & 0.2749  & 0.2393  &  0.2119 \\
\hline
\end{array}
\]
As we see, the necessary condition \eqref{sigmampq}---valid for \emph{any} non-negative matrix---already 
implies that there are positive \emph{lower} bounds on $\nu$, although these bounds are not optimal. 

It is possible to sharpen the lower bounds in \eqref{nustartpq} by making use of some more specific results. 
We know in addition that the matrix $M$ is \emph{circulant}, which leads us to the realm of 
\emph{structured non-negative inverse eigenvalue problems}. For example, 
the spectra of non-negative circulant matrices have been characterized (with a necessary \emph{and} sufficient condition) in \cite[Theorem 10]{rojosoto}. From this theorem we get (still for $m=5$ and $\te=1$) the lower bound
\[
\nu\ge 3.9173.
\]
As we will see, the precise lower bound for this matrix---according to our Theorem~\ref{thm2} with $k=2$ and $\te=1$---is
\[
\nu\ge \nu_R(2,1)\approx 4.4111.
\]
\end{remark}

\begin{remark}
It is not restrictive to assume $a>0$ in \eqref{advection}. If we assumed $a<0$ instead, then the results of Theorems~\ref{thm1} and \ref{thm2} would remain valid (together with Figures~\ref{fig_variousk}--\ref{fig_boundary}, for example), with all the arguments in their proofs being essentially the same.  For example, as we will see in Section~\ref{section3}, the non-negativity of (the first row of) matrix $M$ is governed by the elements $M_{1,1}$ and $M_{1,m}$ for $a>0$ and $m$ odd---this would change to elements $M_{1,1}$ and $M_{1,2}$ for $a<0$ and $m$ odd.
\end{remark}


\section{Second-order centered discretization in space and \texorpdfstring{$\theta$}{}-method in time---algebraic characterization of the entries of the full discretization matrix}\label{section3}

The results of Section~\ref{sectiondiscFourier} are based on the eigendecomposition of the full discretization matrix $M=\cF R(\nu \Lambda) \cF^*$. In this section, instead of using trigonometric functions, we give an algebraic description of the matrix entires by exploiting the relation $M=R(\nu L)$ in \eqref{Mdualdef} with $L$ defined in \eqref{Ldef}.
Explicitly, this means 
\begin{equation}\label{Mdef}
M(m,\te,\nu)=(I-\te\nu L)^{-1}(I+(1-\te)\nu L)\in\mathbb{R}^{m\times m},
\end{equation}
but the dependence of $M$ on its parameters will often be suppressed. 

It is trivial that for $\te=0$ we have $M(m,0,\nu)=I+\nu L$, hence $M\ge 0$ cannot hold for any $\nu>0$. The case $m=2k$ has been discussed in Section~\ref{sectioncentered}. Thus, throughout the rest of this section, we can assume that
\begin{equation}\label{genassump}
\boxed{ 
m=2k+1\quad (k\in\nplus), \ \ \ \nu>0 \text{\ \  and\ \  } 0<\te\le 1.}
\end{equation}

\subsection{Explicit description of the matrix entries for odd values of \texorpdfstring{$m$}{}}\label{explsect}


To illustrate the structure of $M$, we present its first row (as a vector, and with the common denominator of the entries in front of it) for the smallest values of $m$. 
\begin{example}\label{example1} 
For $m=3$ the first row of \eqref{Mdef} is 
\[
\frac{1}{{3 \theta ^2 \nu ^2}/{4}+1}\left(\frac{3 \theta ^2 \nu ^2}{4}-\frac{\theta  \nu ^2}{2}+1,\frac{\theta  \nu ^2}{4}+\frac{\nu }{2},\frac{\theta  \nu ^2}{4}-\frac{\nu }{2}\right),
\]
while for $m=5$ we have 
\[
\frac{1}{{5 \theta ^4 \nu ^4}/{16}+{5 \theta ^2 \nu ^2}/{4}+1}\left(\frac{5 \theta ^4 \nu ^4}{16}-\frac{\theta ^3 \nu ^4}{4}+\frac{5 \theta ^2 \nu ^2}{4}-\frac{\theta  \nu ^2}{2}+1,\right.
\]
\[
\left.\frac{\theta ^3 \nu
   ^4}{16}+\frac{\theta ^2 \nu ^3}{4}+\frac{\nu }{2},\frac{\theta ^3 \nu ^4}{16}-\frac{\theta ^2 \nu ^3}{8}+\frac{\theta  \nu ^2}{4},\frac{\theta ^3
   \nu ^4}{16}+\frac{\theta ^2 \nu ^3}{8}+\frac{\theta  \nu ^2}{4},\frac{\theta ^3 \nu ^4}{16}-\frac{\theta ^2 \nu ^3}{4}-\frac{\nu }{2}\right).
\]

\end{example}

Each element of $M$ is a rational function in the variables $\te$ and $\nu$. From \eqref{Mdef} it is clear that 
\begin{equation}\label{M1jPjkDk}
M_{1,j}=\frac{\cP_{j,k}(\te,\nu) }{\cD_k(\te,\nu)}\quad\quad(j=1,2,\ldots,2k+1),
\end{equation}
where $\cP_{j,k}$ and $\cD_{k}$ are certain bivariate polynomials in $\te$ and $\nu$, and 
\begin{equation}\label{dendet}
\cD_k:=\det\left(I_{(2k+1)\times(2k+1)}-\te\nu L_{(2k+1)\times(2k+1)}\right).
\end{equation}
\begin{remark}
The subscripts of $\cP_{j,k}$ thus refer to the position of the polynomial within the first row of $M$, and the size of $M\in\mathbb{R}^{(2k+1)\times(2k+1)}$, respectively.
\end{remark}

The key to describing $M$ algebraically is the observation that the polynomials $\cP_{j,k}$ and $\cD_{k}$ satisfy certain low-order linear recursions with constant coefficients. As already indicated by Section~\ref{sectioncentered}, the leftmost entry ($j=1$) behaves differently than the rest ($2\le j\le 2k+1$).

\begin{remark}
 \textit{Mathematica}'s  {\tt{FindLinearRecurrence}} command proved to be an efficient tool for discovering these linear recursions. 
\end{remark}
First, let us introduce some new variables. On the one hand, as suggested by Example \ref{example1}, it seems convenient to set
\[
\mu:=\te^2\nu^2>0.
\]
Then, due to the sign assumptions, $\sqrt{\mu}=\te\nu$. On the other hand, as we will soon see, the polynomial 
\[
\kappa ^2-\kappa  \left(1+\frac{\mu }{2}\right)+\frac{\mu ^2}{16}
\]
will appear as a (factor of a) characteristic polynomial, and its roots are
\begin{equation}\label{kappa12}
\kappa_{1,2}=\frac{2+\mu \pm 2 \sqrt{\mu +1}}{4}=\left(\frac{\sqrt{1+\mu}\pm 1}{2}\right)^2.
\end{equation}
This motivates us to introduce yet another variable, which will further simplify our exposition. We set
\begin{equation}\label{ydef01}
y:=\frac{\sqrt{1+\mu}-1}{\sqrt{\mu}}=\frac{\sqrt{1+\te^2\nu^2}-1}{\te\nu}\in (0,1).
\end{equation}
It is seen that the transformation
\[
(0,+\infty)\ni\mu \longleftrightarrow y\in(0,1)
\]
is a bijection. Moreover, the following (inverse) relations 
\[
\mu=\left(\frac{2y}{1-y^2}\right)^2,
\]
\[\mu y^2=2+\mu-2\sqrt{1+\mu},\]
\[\mu/y^2=2+\mu+2\sqrt{1+\mu},\]
and
\begin{equation}\label{fromytonu}
\nu=\frac{2y}{1-y^2}\cdot\frac{1}{\te}
\end{equation}
are easily verified. We can now start describing the entries of the first row of $M$.
\begin{remark}
Although the expressions $\cP_{j,k}$ and $\cD_{k}$ will become in general rational functions in the variable $y$, we still call them polynomials (referring to their structure in the original variables $\te$ and $\nu$). 
\end{remark}
$\bullet$ The polynomials $\cD_k$. By carrying out some determinant expansions, we see that the determinants \eqref{dendet} obey the second-order parametric recursion
\begin{equation}\label{Dkrec}
\cD_{k+2}=\left(1+\frac{\mu}{2}\right)\cD_{k+1}-\frac{\mu^2}{16}\cD_k
\end{equation}
with initial conditions
\[
\cD_1=1+\frac{3\mu}{4}, \quad \cD_2=1+\frac{5\mu}{4}+\frac{5\mu^2}{16}
\]
(cf.~Example \ref{example1}). After solving this recursion, we obtain
\[
\cD_k=\left(\frac{\sqrt{1+\mu}+1}{2}\right)^{2 k+1}- \left(\frac{\sqrt{1+\mu}-1}{2}\right)^{2 k+1},
\]
which, in terms of the variable $y$, becomes
\begin{equation}\label{expldet}
\cD_k=\frac{1-y^{4 k+2}}{\left(1-y^2\right)^{2 k+1} }.
\end{equation}

$\bullet$  The polynomials $\cP_{1,k}$. They satisfy the recursion 
\[
\cP_{1,k+2}=\left(1+\frac{\mu}{2}\right)\cP_{1,k+1}-\frac{\mu^2}{16}\cP_{1,k},
\]
that is, with coefficients being the same as in \eqref{Dkrec}, but with initial conditions
\[
\cP_{1,1}=1+\frac{3 \mu }{4}-\frac{\mu /\theta}{2  }, \quad \cP_{1,2}=1+\frac{5 \mu }{4}+\frac{5 \mu ^2}{16}-\frac{\mu/\theta }{2  }-\frac{\mu ^2/\theta}{4  }
\]
(cf.~Example \ref{example1}). By solving this recursion, we derive that
\begin{equation}\label{P1kexpl}
\cP_{1,k}=\frac{\Pol}{  \left(1+y^2\right)\left(1-y^2\right)^{2 k+1}\theta},
\end{equation}
where the numerator is 
\begin{equation}\label{poldef}
\Pol:=-\theta  y^{4 k+4}-(\theta -2) y^{4 k+2}+(\theta -2) y^2+\theta.
\end{equation}
\begin{remark}
Here, the subscript $L$ stands for \emph{leftmost}. This polynomial will play a special role in the next section.
\end{remark}

$\bullet$  The polynomials $\cP_{2,k}$. They satisfy a third-order recursion in the variable $k$,
\begin{equation}\label{cP2rec}
\cP_{2,k+3}=\left(1+\frac{3\mu}{4}\right)\cP_{2,k+2}-\left(\frac{\mu }{4}+\frac{3 \mu ^2}{16}\right)\cP_{2,k+1}+\frac{\mu ^3}{64}\cP_{2,k},
\end{equation}
with initial conditions 
\[
\cP_{2,1}=\left(\frac{1}{2}+\frac{\sqrt{\mu }}{4}\right) \nu, \quad \cP_{2,2}=\left(\frac{1}{2}+\frac{\mu }{4}+\frac{\mu ^{3/2}}{16}\right) \nu,
\]
\[
\cP_{2,3}=\left(\frac{1}{2}+\frac{\mu }{2}+\frac{3 \mu ^2}{32}+\frac{\mu ^{5/2}}{64}\right) \nu.
\]
The  characteristic polynomial of recursion \eqref{cP2rec} is 
\[
\kappa ^3-\kappa ^2 \left(1+\frac{3 \mu }{4}\right)+\kappa  \left(\frac{\mu }{4}+\frac{3 \mu ^2}{16}\right)-\frac{\mu ^3}{64}=\left(\kappa-\frac{\mu }{4}\right)\left(\kappa ^2-\kappa  \left(1+\frac{\mu }{2}\right)+\frac{\mu ^2}{16}\right),
\]
hence the characteristic roots are $\kappa_{1,2}$ as in \eqref{kappa12}, and $\kappa_3={\mu }/{4}$. Based on this, one easily obtains the explicit solution as
\begin{equation}\label{cP2kexpl}
\cP_{2,k}=\frac{\nu  \left(1-y^2\right)^{1-2 k} \left(1+y^{2 k-1}+y^{2 k+1}-y^{4 k}\right)}{2 \left(1+y^2\right)}.
\end{equation}

$\bullet$  The polynomials $\cP_{3,k}$. They satisfy the same third-order recursion in the variable $k$ as \eqref{cP2rec},
\[
\cP_{3,k+3}=\left(1+\frac{3\mu}{4}\right)\cP_{3,k+2}-\left(\frac{\mu }{4}+\frac{3 \mu ^2}{16}\right)\cP_{3,k+1}+\frac{\mu ^3}{64}\cP_{3,k},
\]
but with initial conditions 
\[
\cP_{3,1}=\left(-\frac{1}{2}+\frac{\sqrt{\mu }}{4}\right) \nu, \quad \cP_{3,2}=\left(\frac{\sqrt{\mu} }{4}-\frac{\mu}{8}+\frac{\mu ^{3/2}}{16}\right) \nu,
\]
\[
\cP_{3,3}=\left(\frac{\sqrt{\mu }}{4}-\frac{\mu ^2}{32}+\frac{3 \mu ^{3/2}}{16}+\frac{\mu ^{5/2}}{64}\right) \nu.
\]
The explicit solution of this recursion is
\begin{equation}\label{cP3kexpl}
\cP_{3,k}=\frac{\nu  \left(1-y^2\right)^{1-2 k} \left(y-y^{2 k-2}+y^{2 k+2}+y^{4 k-1}\right)}{2 \left(1+y^2\right)}.
\end{equation}
\begin{remark}
We note that, for any \emph{fixed} $j\ge 2$, the polynomials $\cP_{j,k}$ satisfy the same third-order recursion \eqref{cP2rec} in the variable $k$, with triplets of initial conditions depending on $j$. However, we cannot use this approach to proceed, since setting up the initial conditions would require, among others, the knowledge of the  polynomials $\cP_{j,1}$ (for $j=2, 3$),  $\cP_{j,2}$ (for $j=4, 5$), $\cP_{j,3}$ (for $j=6, 7$), and so on. 
\end{remark}

$\bullet$  The polynomials $\cP_{j,k}$ ($4\le j\le 2k+1$, $k\ge 2$). They satisfy the following second-order recursion in the variable $j$ when $k$ is \emph{fixed} (hence having only finitely many terms for a particular $k$):  
\[
\cP_{j+2,k}=-\frac{2}{\sqrt{\mu }}\cP_{j+1,k}+\cP_{j,k}.
\]
For the initial conditions of this final recursion, we use the general forms of $\cP_{2,k}$ and $\cP_{3,k}$ in \eqref{cP2kexpl} and \eqref{cP3kexpl} to get for any $k\ge 1$ and $2\le j\le 2k+1$ that 
\begin{equation}\label{Pjkexpl}
\cP_{j,k}=\frac{\nu  \left(1-y^2\right)^{1-2 k}}{2 \left(1+y^2\right)}P_{j,k}(y), 
\end{equation}
where the polynomials $P_{j,k}$ are defined as
\begin{equation}\label{Pjky}
P_{j,k}(y)\coloneqq(-1)^{j-1} y^{4 k+2-j}+y^{2 k-1+j}+(-1)^j y^{2 k+1-j}+y^{j-2}.
\end{equation}
As a special case, we set
\[
\Por:=P_{2k+1,k}(y),
\]
in other words we have
\begin{equation}\label{pordef}
\Por=y^{4 k}+y^{2 k+1}+y^{2 k-1}-1,
\end{equation}
where the subscript $R$ stands for \textit{rightmost}.

\begin{remark}
As a by-product, we have obtained the following set of identities by comparing the trigonometric and algebraic representations presented so far. They are also interesting from a structural point of view: although the number of terms in the trigonometric sums increases as $k$ gets larger, the polynomials in $y$ are sparse polynomials (also known as lacunary polynomials or fewnomials)---the number of terms does not increase as the polynomial degree increases.
\end{remark}
\begin{corollary} \label{corollary1}
With $M$ defined in \eqref{Mdef}, $\te>0$, $\nu>0$, $k\in\nplus$, $y=\frac{\sqrt{1+\te^2\nu^2}-1}{\te\nu}$, and $\xi_\ell  = \frac{2\pi(\ell-1)}{2k+1}$, we have that
\[
 \frac{1}{2k+1} \sum_{\ell=1}^{2k+1} \frac{1+\imath(1-\theta)\nu \sin(\xi_\ell)}{1-\imath\theta\nu  \sin(\xi_\ell)}=
\]
\[
M_{1,1}=\frac{\cP_{1,k}}{\cD_k}=
\]
\[
\frac{-\theta  y^{4 k+4}-(\theta -2) y^{4 k+2}+(\theta -2) y^2+\theta}{  \left(1+y^2\right)\left(1-y^{4 k+2}\right)\theta}.
\]
Moreover, for $j=2, 3, \ldots, 2k+1$ we have that
\[
\frac{1}{2k+1} \sum_{\ell=1}^{2k+1} \frac{1+\imath(1-\theta)\nu \sin(\xi_\ell)}{1-\imath\theta\nu \sin(\xi_\ell)}\exp\left(-\imath(j-1)\xi_\ell\right)=
\]
\[
M_{1,j}=\frac{\cP_{j,k}}{\cD_k}=
\]
\[
\frac{\nu  \left(1-y^2\right)^2 }{2 \left(1+y^2\right) \left(1-y^{4 k+2}\right)}\Big((-1)^{j-1} y^{4 k+2-j}+y^{2 k-1+j}+(-1)^j y^{2 k+1-j}+y^{j-2}\Big).
\]
In particular,
\[
\prod_{\ell=1}^{2k+1} (1-\imath\theta\nu \sin(\xi_\ell))=\cD_k=
\]
\[
\left(\frac{\sqrt{1+\te^2\nu^2}+1}{2}\right)^{2 k+1}- \left(\frac{\sqrt{1+\te^2\nu^2}-1}{2}\right)^{2 k+1}=
\frac{1-y^{4 k+2}}{\left(1-y^2\right)^{2 k+1} }.
\]

\end{corollary}

\subsection{Non-negativity of the matrix entries for odd values of \texorpdfstring{$m$}{}}\label{nonnegsect}

In this section we present a detailed description of the non-negativity properties of the matrix $M$, thanks to the explicit forms for the entries $M_{1,j}$ obtained in Section \ref{explsect}. Throughout this section we still assume \eqref{genassump}.

By taking into account \eqref{M1jPjkDk}, \eqref{expldet}, \eqref{P1kexpl}, \eqref{poldef}, \eqref{Pjkexpl}, \eqref{Pjky}, and the fact that now $y\in(0,1)$ (see \eqref{ydef01}), the following corollary is evident.
\begin{corollary}\label{cor2} For a given pair $(\te,\nu)$
\[
M_{1,1}(2k+1,\te,\nu)\ge 0 \quad \Longleftrightarrow \quad \Pol\ge 0\quad\text{ (see } \eqref{poldef}\text{)},
\]
and for any $2\le j\le 2k+1$
\[
M_{1,j}(2k+1,\te,\nu)\ge 0 \quad \Longleftrightarrow \quad P_{j,k}(y)\ge 0\quad\text{ (see } \eqref{Pjky}\text{)}.
\]
\end{corollary}
\begin{figure}
\begin{center}
\includegraphics[width=0.5\textwidth]{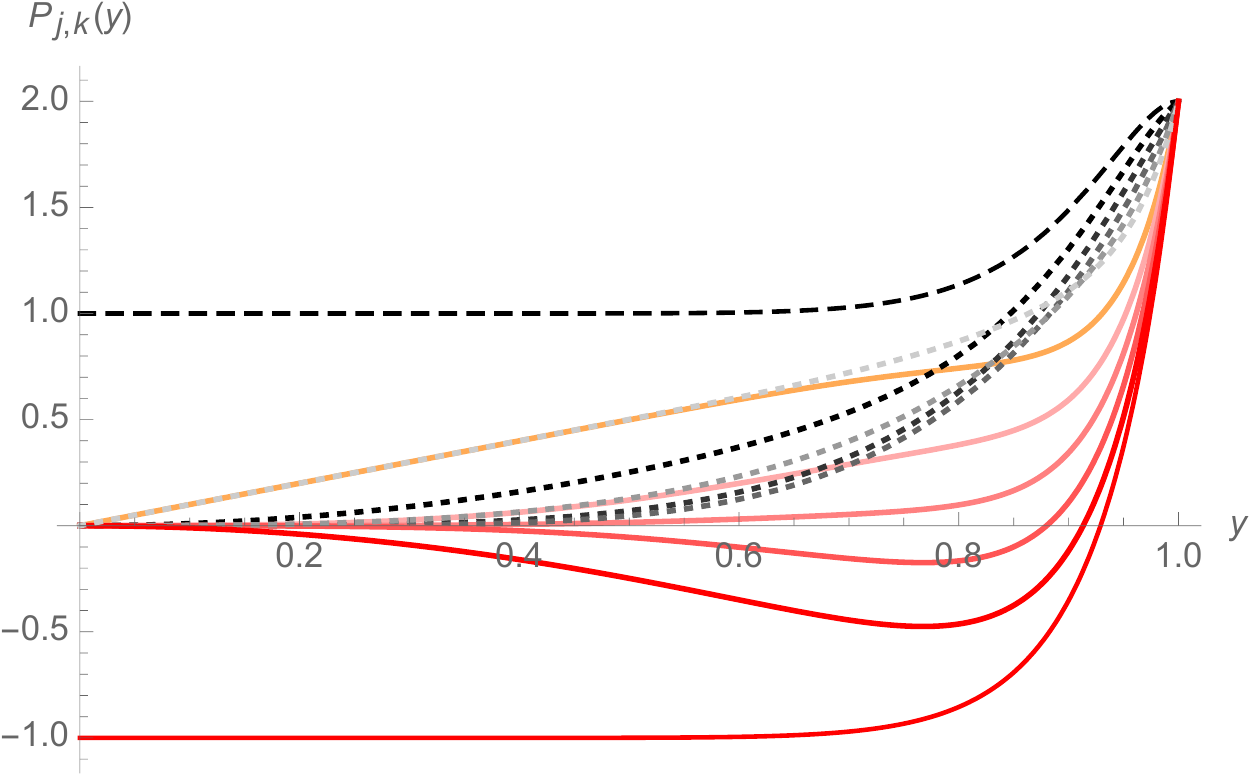}
\caption{The typical behavior of the polynomials $P_{j,k}$ appearing in Corollary \ref{cor2} for $2\le j\le 2k+1$ and $k$ fixed: curves in shades of gray (or black) correspond to even $j$, while curves in shades of red (or orange) correspond to odd $j$ indices. Based on this figure, one can make the following observations. On the one hand, for each fixed and even $j$, $P_{j,k}$ is strictly increasing in $y$; however, for any fixed $y\in(0,1)$, $P_{j,k}$ is in general not monotone in its even index $j$. On the other hand, for each fixed and odd $j$, $P_{j,k}$ is in general not monotone in $y$; however, for any fixed $y\in(0,1)$, $P_{j,k}$ is strictly decreasing in its odd index $j$.}\label{fig_excepttopleft}
\end{center}
\end{figure}
The following lemma proves some of the observations about the polynomials $P_{j,k}$ suggested by
Figure~\ref{fig_excepttopleft} for even and odd indices $2\le j\le 2k+1$.
\begin{lemma}\label{lem2} Let us fix $y\in(0,1)$ arbitrarily. Then\\
\indent $\bullet$ for any $1\le\ell\le k$, $P_{2\ell,k}(y)>0$;\\
\indent $\bullet$ for any $2\le\ell\le k$, $P_{2\ell+1,k}(y)<P_{2\ell-1,k}(y)$. 
\end{lemma}
\begin{proof} For the even indices, we have
\[
P_{2\ell,k}(y)=y^{2 k-2 l+1}(1-y^{2 k+1})+y^{2 k+2 l-1}+y^{2 l-2}>0,\]
while for the odd indices,
\[
P_{2\ell+1,k}(y)-P_{2\ell-1,k}(y)=-(1 - y^2)\Big(y^{2 k+2 l-2}+y^{2 l-3}+y^{2 k-2 l}(1-y^{2k+1})\Big)<0.
\] 
\end{proof}
By combining Corollary \ref{cor2} and Lemma \ref{lem2}, we have obtained the following result, expressing the fact that the non-negativity of $M(2k+1,\te,\nu)$ is determined only by the polynomials appearing in the numerators of its top left and top right entries.
\begin{corollary}\label{cor3} For a given pair $(\te,\nu)$
\[
M_{1,1}(2k+1,\te,\nu)\ge 0 \quad \Longleftrightarrow \quad \Pol\ge 0 \quad \text{(see } \eqref{poldef}\text{)},
\]
and
\[
M_{1,j}(2k+1,\te,\nu)\ge 0 \text{  for each } 2\le j\le 2k+1 \quad \Longleftrightarrow \quad
\Por\ge 0\quad\text{(see }\eqref{pordef}\text{)}.
\]
\end{corollary}
The non-negativity of $M(2k+1,\te,\nu)$ has therefore been reduced to studying the simultaneous non-negativity of two parametric polynomials, $P_{L,k,\te}$ and $P_{R,k}$, over the $y$-interval $(0,1)$. The content of Lemmas \ref{lem3} and \ref{lem4} is illustrated by Figure~\ref{fig_someplpr}.
\begin{lemma}[about the sign of $\Por$]\label{lem3}
Let us fix $k$ arbitrarily, and recall that by definition $\Por=y^{4 k}+y^{2 k+1}+y^{2 k-1}-1$. Then there is a unique $y\in(0,1)$ such that $\Por=0$.\\ 
Let 
\begin{equation}\label{yrdef}\yr \text{ denote this root.}\end{equation} 
Then $\Por<0$ for $y\in(0,\yr)$, and $\Por>0$ for $y\in(\yr,1)$.\\
Moreover, $\yr<y_R(k+1)$, $\lim_{k\to+\infty} \yr=1$, and
\begin{equation}\label{yrasympt}
\left(\sqrt{2}-1\right)^{\frac{1}{2 k-1}}<\yr < \left(\sqrt{2}-1\right)^{\frac{1}{2 k+1}}.
\end{equation}
\end{lemma}
\begin{proof}
For fixed $k$, the continuous function $y\mapsto\Por=y^{4 k}+y^{2 k+1}+y^{2 k-1}-1$ is strictly increasing, 
$P_{R,k}(0)<0$ and $P_{R,k}(1)>0$, hence there is a unique root. This root is strictly increasing in $k$, because the function $k\mapsto\Por$ is strictly decreasing for fixed $y\in(0,1)$. Finally notice that 
$y^{4 k}+y^{2 k+1}+y^{2 k-1}-1=0$ is equivalent to $\left(y^{2 k-1}+1\right) \left(y^{2 k+1}+1\right)=2$,
and for any $y\in(0,1)$ one has
\[
\left(y^{2 k+1}+1\right)^2<\left(y^{2 k-1}+1\right) \left(y^{2 k+1}+1\right)<\left(y^{2 k-1}+1\right)^2.
\]
From this we easily get \eqref{yrasympt} and also the limit of $\yr$ ($k\to +\infty$).
\end{proof}
\begin{remark}
The asymptotic series (as $k\to+\infty$) of both bounds in \eqref{yrasympt} has the form
\[
1+\frac{\ln \left(\sqrt{2}-1\right)}{2 k}+{\cal{O}}\left(\frac{1}{k^2}\right)\approx 
1-\frac{0.44069}{k}+{\cal{O}}\left(\frac{1}{k^2}\right).
\]
\end{remark}

\begin{lemma}[about the sign of $\Pol$]\label{lem4} Let us fix $k$ arbitrarily, and recall that by definition 
$\Pol=-\theta  y^{4 k+4}-(\theta -2) y^{4 k+2}+(\theta -2) y^2+\theta$.\\
(i) Suppose that $\frac{2k}{2k+1}\le\te\le1$. Then, for any $y\in(0,1)$, $\Pol>0$.\\
(ii) Suppose now that $0<\te<\frac{2k}{2k+1}$. Then there is a unique $y\in(0,1)$ such that $\Pol=0$.\\ 
Let 
\begin{equation}\label{yldef}\yl \text{ denote this root.}\end{equation} 
Then $\Pol>0$ for $y\in(0,\yl)$, and $\Pol<0$ for $y\in(\yl,1)$.\\
Moreover, on the one hand, for fixed $0<\te<\frac{2k}{2k+1}$, the function $k\mapsto\yl$ is strictly decreasing, and $\lim_{k\to+\infty} \yl=\sqrt{\frac{\te}{2-\te}}\in(0,1)$.\\
On the other hand, for fixed $k\in\nplus$, the function $\left(0,\frac{2k}{2k+1}\right)\ni\te\mapsto\yl$ is strictly increasing, and we have the one-sided limits $\lim_{\te\to0+0} \yl=0$ and $\lim_{\te\to\frac{2k}{2k+1}-0} \yl=1$.
\end{lemma}
\begin{proof}
We notice that the expression $\Pol$ is linear in $\te$, so by setting 
\[
\Theta(y,k)\coloneqq\frac{2 y^2 \left(1-y^{4 k}\right)}{\left(1+y^2\right) \left(1-y^{4 k+2}\right)},
\]
we easily get for any $y\in(0,1)$ that
\begin{equation}\label{lem4equi}
\Pol \lesseqqgtr 0\quad \Longleftrightarrow \quad \te\lesseqqgtr\Theta(y,k),
\end{equation}
where the symbol $\lesseqqgtr$ denotes either $<$, or $=$, or $>$ on both sides of the equivalence. It is
seen that for fixed $k$ we have the one-sided limits
\begin{equation}\label{lem4lim}
\lim_{y\to 0+0}\Theta(y,k)=0\quad\text{and}\quad\lim_{y\to 1-0}\Theta(y,k)=\frac{2k}{2k+1}.
\end{equation}
Now we show that the function
\begin{equation}\label{Thetaincr}
(0,1)\ni y\mapsto \Theta(y,k)\quad\text{is strictly increasing.}
\end{equation}
The partial derivative
\[
\partial_y\Theta(y,k)=\frac{4 y \left(1-(2 k+1) y^{4 k}+(2 k+1) y^{4 k+4}-y^{8 k+4}\right)}{\left(1+y^2\right)^2 \left(1-y^{4 k+2}\right)^2}
\]
is positive, if $\widetilde{P}(y,k):=1-(2 k+1) y^{4 k}+(2 k+1) y^{4 k+4}-y^{8 k+4}>0$. But 
\[
\widetilde{P}(0,k)=1\quad\text{and}\quad\widetilde{P}(1,k)=0,
\]
so the positivity of $\widetilde{P}(y,k)$ will follow if we show that $y\mapsto\widetilde{P}(y,k)$ is strictly decreasing. Indeed,
\[
\partial_y\widetilde{P}(y,k)=-4 (2 k+1) y^{4 k-1} \widetilde{Q}(y,k),
\]
where
\[
\widetilde{Q}(y,k):=y^{4 k+4}-(k+1) y^4+k,
\]
hence it is enough to verify $\widetilde{Q}(y,k)>0$. And this is true, since $\widetilde{Q}(0,k)=k$, $\widetilde{Q}(1,k)=0$ and
\[
\partial_y \widetilde{Q}(y,k)=-4 (k+1) y^3 \left(1-y^{4 k}\right)<0.
\]
Now, as \eqref{Thetaincr} has been checked, it is obvious that continuity, \eqref{lem4equi}, \eqref{lem4lim} and \eqref{Thetaincr} imply statement (\textit{i}) of the lemma, and, at the same time, regarding statement (\textit{ii}) of the lemma, the existence of a unique root $\yl\in(0,1)$, the positivity of $P_{L,k,\te}$ on $(0,\yl)$, and the negativity of $P_{L,k,\te}$ on $(\yl,1)$.

We finally discuss the monotonicity and limit properties of the root $\yl$. For fixed $y\in(0,1)$, the function $k\mapsto\Theta(y,k)$ is strictly increasing, since
\[
\Theta(y,k+1)-\Theta(y,k)=\frac{2 \left(1-y^2\right)^2 y^{4 k+2}}{\left(1-y^{4 k+2}\right) \left(1-y^{4
   k+6}\right)}>0.
\]
This implies that, for any fixed $\te\in\left(0,\frac{2k}{2k+1}\right)$, the function $k\mapsto\yl$ is strictly decreasing. Moreover, for fixed $y\in(0,1)$, we see from the definition that $\lim_{k\to+\infty} \Theta(y,k)=\frac{2 y^2}{1+y^2}$,
so, due to \eqref{lem4equi} with ``equality'', one has for fixed $\te\in\left(0,\frac{2k}{2k+1}\right)$ that $y_{\infty}(\te)\coloneqq\lim_{k\to+\infty} \yl$ solves $\te=\frac{2 y_{\infty}(\te)^2}{1+y_{\infty}(\te)^2}$; in other words, $y_{\infty}(\te)=\sqrt{\frac{\te}{2-\te}}\in(0,1)$. To show the validity of the last sentence of the lemma, we fix $k\in\nplus$, and simply take into account again \eqref{lem4equi} with ``equality'', \eqref{lem4lim} and \eqref{Thetaincr}.
\end{proof}
\begin{figure}
\begin{center}
\includegraphics[width=0.5\textwidth]{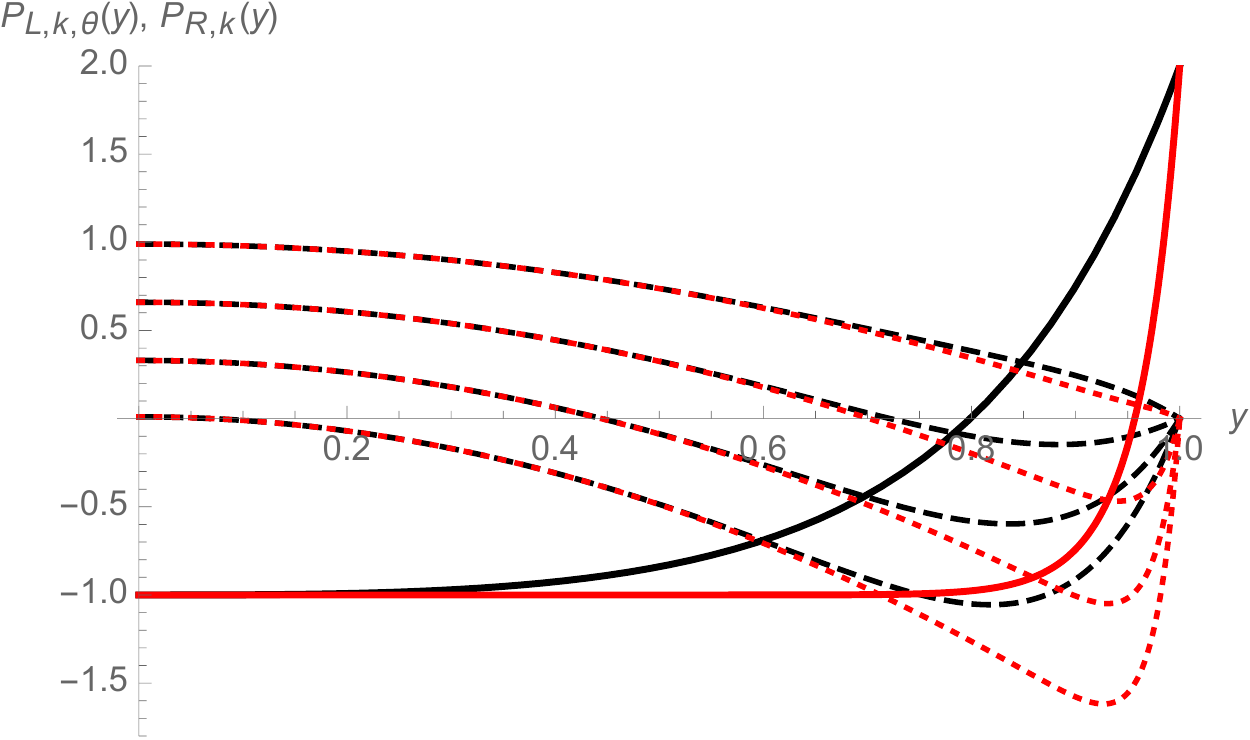}
\caption{The two solid curves show the functions $y\mapsto\Por$ for some $k=k_0$ (solid black) and 
 $k=k_1$ (solid red) with $k_0<k_1$.  The dashed black curves show the functions $y\mapsto\Pol$ for $k=k_0$ and for various values of $\te\in(0,1]$. Finally, the dotted red curves show the functions $y\mapsto\Pol$ for $k=k_1$ and for the same values of $\te\in(0,1]$.}\label{fig_someplpr}
\end{center}
\end{figure}

In order to return to the original variables $(\te,\nu)$ from the variable $y$---based on \eqref{fromytonu},  \eqref{yrdef} and \eqref{yldef}---we define
\begin{equation}\label{nurdef}
\nur:=\frac{2\yr}{1-\yr^2}\cdot \frac{1}{\te},
\end{equation}
and similarly,
\begin{equation}\label{nuldef}
\nul:=\begin{cases}
 \frac{2\yl}{1-\yl^2}\cdot \frac{1}{\te} & \text{for } 0<\te<\frac{2k}{2k+1}\\
 +\infty & \text{for } \frac{2k}{2k+1}\le \te\le 1.
\end{cases}
\end{equation}
The value $+\infty$ is introduced here for convenience so as to make our descriptions shorter.

A reformulation of Corollary \ref{cor3} in terms of the variables $(\te,\nu)$ is given below.
\begin{corollary}\label{cor4}
For any $k\in\nplus$ and $\te\in(0,1]$ we have 
\[
M_{1,1}(2k+1,\te,\nu)\ge 0 \quad \Longleftrightarrow \quad \nu\le\nul,
\] 
and
\[
M_{1,j}(2k+1,\te,\nu)\ge 0 \text{  for each } 2\le j\le 2k+1 \quad \Longleftrightarrow \quad \nu\ge\nur.
\] 
In particular, 
\[
M_{1,j}(2k+1,\te,\nu)\ge 0 \text{  for each } 1\le j\le 2k+1  \quad \Longleftrightarrow \quad \nur\le\nu\le\nul.
\] 
\end{corollary}
\begin{proof} By taking into account Corollary \ref{cor3}, Lemmas \ref{lem3} and \ref{lem4}, and the fact that the map in \eqref{fromytonu} 
\begin{equation}\label{cor4bijection}
(0,1)\ni y\mapsto\frac{2y}{1-y^2}\in(0,+\infty)\text{ is a strictly increasing bijection,}
\end{equation}
we get for fixed $k$ and $\te$ that $\Por\ge 0$ is equivalent to $\nu\ge\nur$, and
$\Pol\ge 0$ is equivalent to $\nu\le\nul$. In particular, due to the definition of $\nul$ in \eqref{nuldef}, this last inequality means that there is no upper bound on $\nu$ for $\frac{2k}{2k+1}\le \te\le 1$.
\end{proof}
Some growth rates, monotonicity and limit properties of $\nur$ and $\nul$---defined in 
\eqref{nurdef}--\eqref{nuldef}---are collected below; 
see also Figures~\ref{fig_variousk} and \ref{fig_boundary}.
\begin{corollary}\label{cor5} (i) For any $k\in\nplus$ and $\te\in(0,1]$, we have $\nur<\nu_R(k+1,\theta)$, and
\begin{equation}\label{cor5lowerupper}
\frac{2 \left(\sqrt{2}+1\right)^{\frac{1}{2 k-1}}}{\left(\sqrt{2}+1\right)^{\frac{2}{2
   k-1}}-1}\cdot\frac{1}{\te}<\nur<\frac{2 \left(\sqrt{2}-1\right)^{\frac{1}{2 k+1}}}{1-\left(\sqrt{2}-1\right)^{\frac{2}{2
   k+1}}}\cdot\frac{1}{\te}.
\end{equation}
The asymptotic series for these lower and upper bounds have the form
\[
\left( \frac{2}{\ln \left(\sqrt{2}+1\right)}k\mp\frac{1}{\ln \left(\sqrt{2}+1\right)}+ {\cal{O}}\left(\frac{1}{k}\right)\right)\cdot\frac{1}{\te},
\]
being approximately
$
\left( 2.26919 k\mp1.13459+ {\cal{O}}\left(\frac{1}{k}\right)\right)\cdot\frac{1}{\te}.
$
In particular, $\lim_{k\to+\infty} \nur=+\infty$.\\
(ii) For fixed $0<\te<\frac{2k}{2k+1}$, $\nul>\nu_L(k+1,\theta)$ (and $\nul =+\infty$ for $\frac{2k}{2k+1}\le\te\le 1$). Finally, for fixed $\te\in(0,1)$, we have the limit
\begin{equation}\label{cor5klim}
\lim_{k\to+\infty} \nul=\frac{1}{1-\theta }\sqrt{\frac{2-\theta }{\theta }},
\end{equation}
and, for fixed $k\in\nplus$, the one-sided limits  
\begin{equation}\label{cor5thetalim}
\lim_{\te\to 0+0} \nul=+\infty=\lim_{\te\to\frac{2k}{2k+1}-0} \nul.
\end{equation}
\end{corollary}
\begin{proof} (\textit{i}) The monotonicity of $\nur$ in $k$ for fixed $\te$ follows from the monotonicity of 
$\yr$ in Lemma \ref{lem3} together with \eqref{cor4bijection}, and inequality \eqref{cor5lowerupper} is 
just \eqref{yrasympt} under the transformation \eqref{cor4bijection}.\\
(\textit{ii}) We similarly obtain the monotonicity of $\nul$ in $k$ for fixed $\te$, and the limit \eqref{cor5klim} from Lemma \ref{lem4} via \eqref{cor4bijection}, by also noting that
\[
\frac{2\sqrt{\frac{\te}{2-\te}}}{1-\left(\sqrt{\frac{\te}{2-\te}}\right)^2}\cdot \frac{1}{\te}=\frac{1}{1-\theta }\sqrt{\frac{2-\theta }{\theta }}.
\]
As for the $\te\to\frac{2k}{2k+1}-0$ limit in \eqref{cor5thetalim}, we know from Lemma \ref{lem4} that $\yl\to 1$ (from below), and $\lim_{y\to 1-0}\frac{2y}{1-y^2}=+\infty$, hence $\nul\to+\infty$ when $\te\to\frac{2k}{2k+1}-0$.

One needs to take care only when evaluating the $\te\to 0+0$ limit in \eqref{cor5thetalim} for fixed $k\in\nplus$,  since $\frac{2\yl}{1-\yl^2}\to 0$ and $\frac{1}{\te}\to+\infty$ in \eqref{nuldef} when $\te\to0+0$. But 
the monotonicity of $\nul$ in $k$ for fixed $\te$, and \eqref{cor5klim} imply for any $k$ and $\te\in(0,1)$ that 
\[
\nul\ge\frac{1}{1-\theta }\sqrt{\frac{2-\theta }{\theta }},
\]
and the right-hand side here tends to $+\infty$ as $\te\to 0+0$.
\end{proof}

\begin{figure}
\begin{center}
\includegraphics[width=0.48\textwidth]{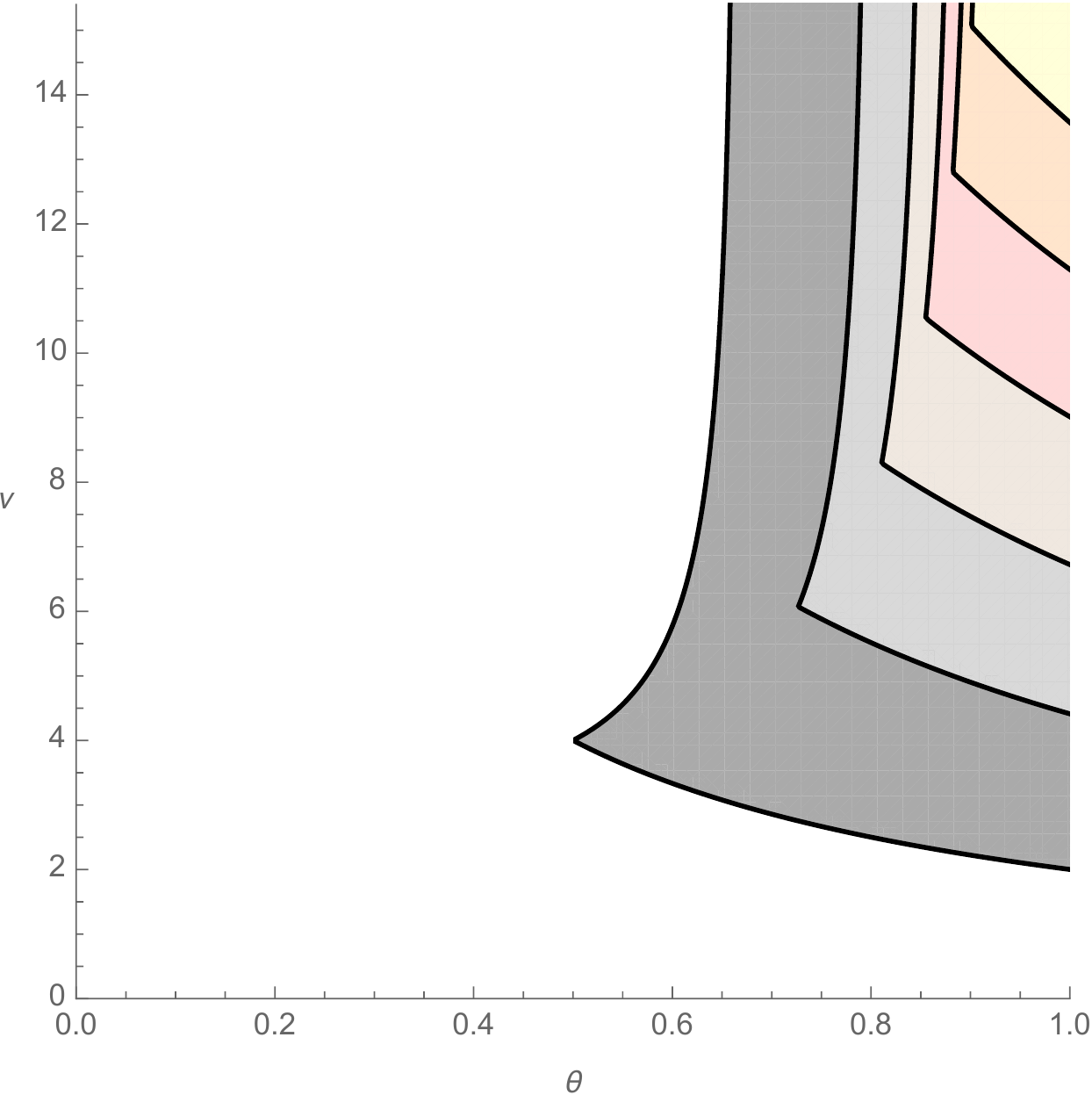}
\caption{The parameter regions in the $(\te,\nu)$ parameter plane ensuring $M(2k+1,\te,\nu)\ge 0$ for $k=1, 2, \ldots, 6$ (different values of $k$ are represented by different colors). The regions continue to extend to infinity ``upward'', but ``shrink'' in the horizontal direction as $k$ is increased.}\label{fig_variousk}
\end{center}
\end{figure}

\begin{figure}
\begin{center}
\includegraphics[width=0.48\textwidth]{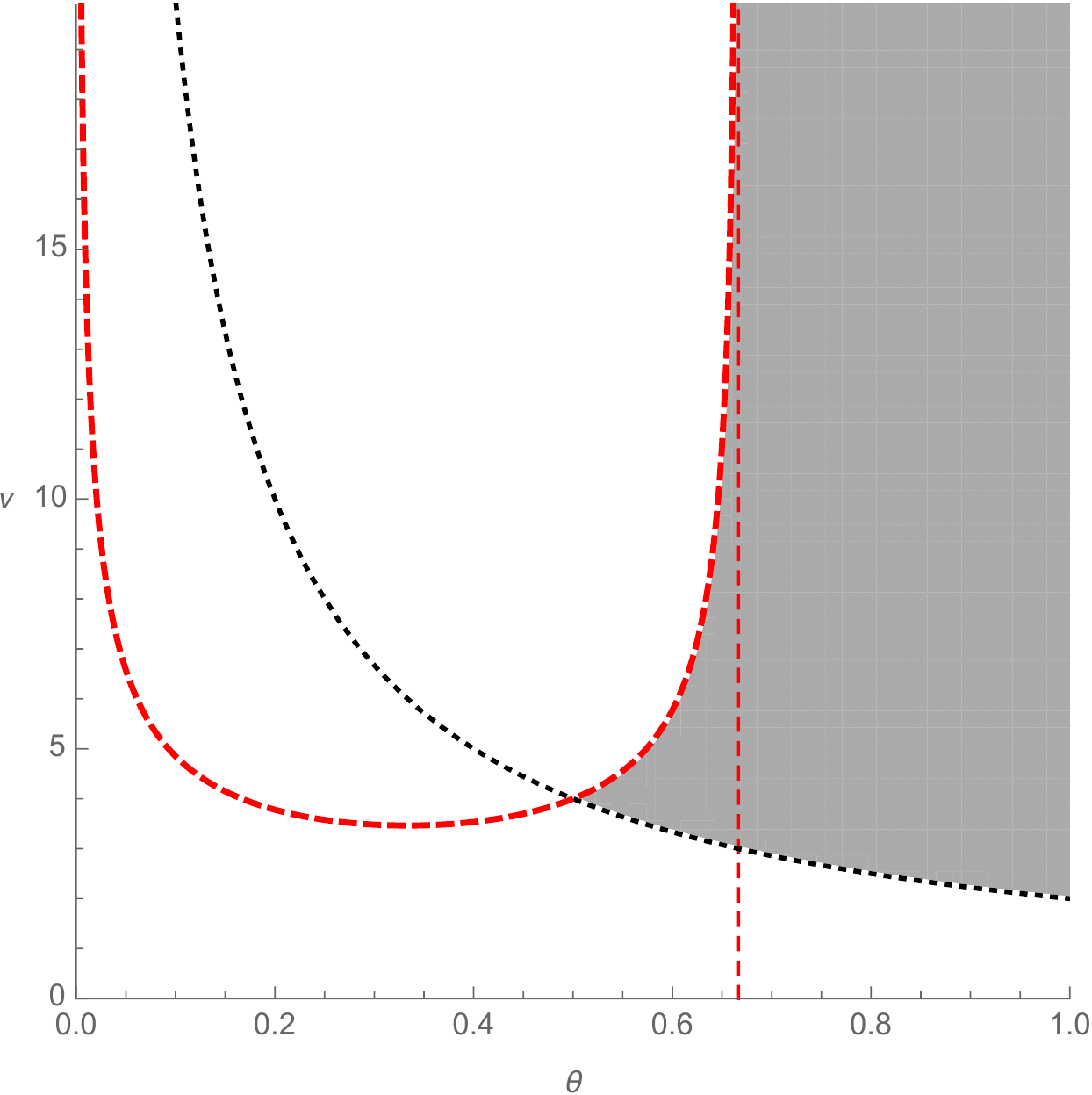}
\caption{A typical shaded region in Figure~\ref{fig_variousk} for which $M(2k+1,\te,\nu)\ge 0$; in this particular case, for $k=1$. The gray region is described by the inequalities $\nur\le\nu\le\nul$. The black dotted curve represents the function $\te\mapsto\nur$, while the red dashed curve is the function
$\te\mapsto\nul$, having a vertical asymptote at $\te=2k/(2k+1)$.}\label{fig_boundary}
\end{center}
\end{figure}

The following result explains why the ``left half'' of Figure~\ref{fig_variousk} is ``empty'' (cf.~Corollary \ref{cor4})---the result is non-trivial, since for fixed $k$, $\lim_{\te\to 0+0} \nul=+\infty=\lim_{\te\to 0+0} \nur$ (cf.~Figure~\ref{fig_boundary}).
\begin{lemma}\label{lem5} For any $k\in\nplus$ there is a unique $\te_k\in\left[\frac{1}{2},\frac{2k}{2k+1}\right)$  such that 
\[
\nur = \nul\ \ \  (\text{for } \te\in(0,1])\quad \Longleftrightarrow \quad \te= \te_k.
\]
This $\te_k$ also satisfies
\[
\nur < \nul\quad \Longleftrightarrow \quad \te> \te_k.
\]
Moreover, the sequence $\te_k$ is strictly increasing in $k$, and  $\te_1=\frac{1}{2}$. In particular, for any $k\in\nplus$ and $\te\in\left(0,\frac{1}{2}\right)$ we have 
\[
 \nur>\nul.
\] 
\end{lemma}
\begin{proof} Let us fix $k$. Due to \eqref{nurdef}--\eqref{nuldef}, $\nur$ is finite but $\nul$ is infinite for any $\te\in\left[\frac{2k}{2k+1},1\right]$, so $\nur = \nul$ cannot hold. For $\te\in\left(0,\frac{2k}{2k+1}\right)$, by also using \eqref{cor4bijection}, we have 
\[
\nur = \nul\quad \Longleftrightarrow \quad \frac{2\yr}{1-\yr^2}=\frac{2\yl}{1-\yl^2} \quad \Longleftrightarrow
\quad \yr=\yl.
\]
Here $\yr$ is independent of $\te$, and $k$ is fixed, so by definitions \eqref{yrdef} and \eqref{yldef} this means that $\te$ must be chosen in a way such that $P_{L,k,\te}(\yr)=0$. By using the notation introduced in the proof of Lemma \ref{lem4}, this is equivalent to $\te=\Theta(\yr,k)=:\te_k$. Hence, if $\nur = \nul$ holds for some $\te\in(0,1]$, then $\te=\te_k\in\left(0,\frac{2k}{2k+1}\right)$. Now, by Lemma \ref{lem3}, we have $\yr<y_R(k+1)$, and $\Theta$ is strictly increasing in its first argument (see \eqref{Thetaincr}), so the sequence $\te_k$ is also strictly increasing. 

The same monotonicity argument shows that $\nur < \nul$ holds for some $\te\in(0,1]$ if and only if $\te>\te_k$
(see \eqref{lem4equi} and the characterization of  $P_{L,k,\te}>0$ in Lemma \ref{lem4}). 

For $k=1$, one explicitly computes that
\begin{equation}\label{nurl1thetaconcrete}
\nu_R(1,\theta)=\frac{2}{\te}\quad\text{ and }\quad\nu_L(1,\theta)=\frac{2}{\sqrt{\theta  (2-3 \theta)}},
\end{equation}
so $\nu_R(1,\theta)=\nu_L(1,\theta) \Longleftrightarrow \te=\te_1=1/2$. Therefore, we have
$\te_k\in\left(\frac{1}{2},\frac{2k}{2k+1}\right)$ for $k\ge 2$, also implying that for any $k\in\nplus$ and 
$\te\in\left(0,\frac{1}{2}\right)$, we have $\nur>\nul$.
\end{proof}

The following theorem summarizes the results of Section~\ref{section3}.
In the theorem, we assume $k\in\nplus$, $\te\in[0,1]$ and $\nu\in(0,+\infty)$.
\begin{theorem}[About the full discretization matrix corresponding to the $2^{\text{nd}}$-order centered discretization in space and $\theta$-method in time]\label{thm2}
\begin{itemize}\ 
\item[$\bullet$] Fix $0\le\te<1/2$ arbitrarily. Then  $M(2k+1,\te,\nu)\ge 0$ can never hold, i.e.~for any $k\in\nplus$ and $\nu>0$ there is at least one strictly negative entry of the matrix $M$.
\item[$\bullet$] Let $\te=1/2$. Then
\[
M\left(2k+1,\frac{1}{2},\nu\right)\ge 0\quad\Longleftrightarrow\quad k=1\text{ and } \nu=4\quad\text{(see }\eqref{rem10thetahalfmatrix}\text{)}.
\]
\item[$\bullet$] Fix $1/2<\te<1$ arbitrarily. Then there are finitely many values of $k$ for which there exists $\nu>0$ with $M(2k+1,\te,\nu)\ge 0$. For any such value of $k$, the set of admissible values of $\nu$ has the form $\nur\le\nu\le\nul$, with suitable constants $0<\nur\le\nul\le+\infty$ (the possible case $\nul=+\infty$ means that there is no upper but only a lower bound on $\nu$); see also Corollary \ref{cor5}.
\item[$\bullet$] Let $\te=1$. Then for each $k\in\nplus$ there is a constant $\nu_R(k,1)>0$ such that
\[
M(2k+1,1,\nu)\ge 0\quad\Longleftrightarrow\quad \nu\ge\nu_R(k,1).
\]
In addition, $\nu_R(k,1)<\nu_R(k+1,1)$ for any $k$, $\lim_{k\to+\infty} \nu_R(k,1)=+\infty$, and the two-sided estimates in \eqref{cor5lowerupper} with $\te=1$ hold.
\end{itemize}
\end{theorem}
\begin{proof} The case $\te=0$ has already been discussed at the beginning of Section~\ref{section3}.
In general, for $\te\in\left(0,1\right]$, we know from Corollary \ref{cor4} that, for any $k$, the set of $\nu$ values for which $M(2k+1,\te,\nu)\ge 0$ holds has the form $\nur\le\nu\le\nul$. 

The range $\te\in\left(0,\frac{1}{2}\right)$ is covered by Lemma \ref{lem5}. 

For $\te=1/2$, \eqref{nurl1thetaconcrete} shows that for $k=1$ one has $\nu_R(1,1/2)=\nu_L(1,1/2)=4$. 
But for any $k\ge 2$ we know (see Corollary \ref{cor5}) that
\[
\nu_L(k,1/2)<\nu_L(1,1/2)=\nu_R(1,1/2)<\nu_R(k,1/2),
\]
hence $\nu_R(k,1/2)\le\nu_L(k,1/2)$ cannot hold for any $k\ge 2$.

For fixed $1/2<\te<1$, $\nul$ becomes finite for all sufficiently large $k$ (see \eqref{nuldef}). But according to Corollary \ref{cor5}, $\nul$ is decreasing in $k$ for $\te<\frac{2k}{2k+1}$, and 
$\lim_{k\to+\infty} \nur=+\infty$, so the inequality $\nur\le\nul$ can hold only for finitely many values of $k$.

Finally, for $\te=1$, $\nu_L(k,1)=+\infty$ and we can use Corollary \ref{cor5} (\textit{i}) with $\te=1$.
\end{proof}

\begin{remark}
The ``lower left corner point'' of each shaded region in Figure~\ref{fig_variousk} corresponds to a pair 
$(\te,\nu)$  for which $\nu=\nur=\nul$. This means that here the leftmost and the rightmost entries of the first row of $M(2k+1,\te,\nu)$ simultaneously vanish (and the other entries are non-negative). For $k=1$, this happens for $\te=\te_1=1/2$ and $\nu=4$; the corresponding matrix is
\begin{equation}\label{rem10thetahalfmatrix}
M\left(3,\frac{1}{2},4\right)=\left(
\begin{array}{ccc}
 0 & 1 & 0 \\
 0 & 0 & 1 \\
 1 & 0 & 0 \\
\end{array}
\right).
\end{equation}
\end{remark}

\section{Other spatial discretizations}\label{otherdisc}

The spatial semi-discretization considered in Sections~\ref{sectioncentered}--\ref{section3} is not positivity preserving. The same will hold true for each spatial semi-discretization to be investigated in Sections~\ref{highercentered}--\ref{spectrcoll} below: it is well-known \cite[Chapter I,
Theorem 7.2]{hundsdorferverwer} that a linear constant-coefficient system of ordinary differential equations 
\eqref{semi-discrete} is positivity preserving if and only if the
matrix $\frac{a}{\dx}L$ has no negative off-diagonal entries. The violation of this last condition is clear for the matrix 
$L$ in \eqref{Ldef}, for all matrices $L$ in Section~\ref{highercentered}, and also for the ones in Section~\ref{spectrcoll} (due to $L_{1,2}=-L_{2,1}\ne 0$).

However, as we will see, it is again possible to obtain a positivity-preserving full discretization scheme when the spatial
discretizations covered in this section (higher-order centered
spatial discretizations and Fourier spectral collocation methods) are suitably combined  with the $\theta$-method.
\begin{remark}
For $\theta=0$ (that is, when the explicit Euler time discretization is applied), the matrix $M$ in
\eqref{Mdualdef} becomes $M = I + \nu L$, hence $M \ge 0$ cannot hold for any $\nu > 0$ due to the
negative off-diagonal entries of $L$.
Therefore, in what follows, we can assume $\theta \in (0,1]$.
\end{remark}


\subsection{Higher-order centered discretizations in space, \texorpdfstring{$\theta$}{}-method in time}\label{highercentered}
The coefficients of the centered differences can be found, e.g., in \cite{bengt}, from which the
corresponding circulant matrices $L$ describing the spatial discretization can be constructed.
Here, we examine the first few cases.
\begin{itemize}
\item When the stencil width is 3 (implying  $m\ge 3$ and $2^\text{nd}$-order accuracy), the entries on
the central diagonals are $(-1/2, 0, 1/2)$. This is matrix $L$ in \eqref{Ldef} that has been considered in
Sections~\ref{sectioncentered}--\ref{section3}.
\item When the stencil width is 5 (implying  $m\ge 5$ and $4^\text{th}$-order accuracy), the entries on
the central diagonals are $(1/12, -2/3, 0, 2/3, -1/12)$.
\item When the stencil width is 7 (implying $m\ge 7$ and $6^\text{th}$-order accuracy), the entries on the
central diagonals are $(-1/60, 3/20, -3/4, 0, 3/4, -3/20, 1/60)$.
\end{itemize}
In all above central diagonals, the middle $0$ corresponds to the main diagonal.
We can obtain an eigendecomposition \eqref{eigen} for the matrix $L$, with eigenvectors given by
\eqref{eigenvectrors} and eigenvalues $\lambda_\ell = \imath \psi(\xi_\ell)$, where
\begin{align}\label{psi}
	\psi(x) := 2 \sum_{k=1}^{(N-1)/2} C_k \sin(k x) \quad (x\in\rr).
\end{align}
As before, $\xi_\ell$ is defined in \eqref{xildef}, and $C$ is a vector consisting of the last $(N-1)/2$
coefficients of the central diagonals of $L$, with $N$ denoting the stencil width.
For instance, the vector $C$ is equal to $(1/2)$, $(2/3,-1/12)$, and $(3/4,-3/20,1/60)$ for stencil widths
$3$, $5$, and $7$, respectively.

After the matrix $L$ has been chosen, we couple this spatial discretization with the $\theta$-method as time
discretization, and the full discretization matrix $M$ is obtained (see \eqref{Mdualdef} and \eqref{R}).
As seen in Section~\ref{sectioncentered}, the matrix $M$ is a real, circulant matrix so it can be characterized by the entries of its first row,
which take the form
\begin{align}\label{M1j_FD}
	  M_{1,j}  = \frac{1}{m} \sum_{\ell=1}^{m} \frac{\left(1-\theta(1-\theta)\nu^2\psi^2(\xi_\ell)\right)
  \cos((j-1)\xi_\ell) + \nu \psi(\xi_\ell)\sin((j-1)\xi_\ell)}{1+\theta^2\nu^2 \psi^2(\xi_\ell)} \quad (1\le j \le m).
\end{align}
Our computations suggest that the non-negativity properties of the matrix family $M(m,\theta,\nu)$ again
depend on the parity of $m$.

\begin{description}[style=unboxed,leftmargin=0cm]
\item [{Case 1:} $m$ is {even}.]
\item \noindent By using symbolic calculations,  we have found that, for the $4^\text{th}$-order scheme, 
$M(m,\theta,\nu)\ge 0$ cannot hold for any $\theta\in(0,1]$ and $\nu>0$ when $m\in\{6,8,10\}$.
Similarly, for the $6^\text{th}$-order scheme, we checked (again symbolically) that $M(m,\theta,\nu)\ge 0$
does not hold for any $\theta\in(0,1]$ and $\nu>0$ when $m\in\{8,10\}$.
Therefore, positivity preservation is impossible in these cases.

The following proposition extends the above observations for the $4^\text{th}$-order scheme when $m$ is a
general even number---although only for sufficiently large values of $\nu$.
\begin{proposition}\label{m_even_nonpositivityFE}
	Consider the iterative formula \eqref{M} applied to the advection equation \eqref{advection} with periodic
	boundary condition.
	Let the matrix $M_{m \times m}$ result from the $4^\text{th}$-order centered discretization in space with
	$m$ spatial grid points and the $\theta$-method in time.
	Also, let $\nu$ be the CFL number defined in \eqref{nudef}.
	If $m \ge 6$ is \emph{even}, then there exists $\nu_0  > 0$ such that the matrix $M$ has at least one
	negative entry for any $\nu > \nu_0$.
\end{proposition}
\begin{proof}
	We show that $M_{1,m} < 0$ for any $\theta \in (0,1]$ and $\nu > \nu_0$,
	where $\nu_0 > 0$ is a constant depending on $m$ and $\theta$.
	
	Let $j = m$ in \eqref{M1j_FD}, then by using \eqref{angle-identities} and
	$\frac{1}{m}\sum_{\ell=1}^m \cos(\xi_\ell) = 0$ we have
	\begin{align}\label{M1m}
		M_{1,m} &= \frac{1}{m} \sum_{\ell=1}^{m} \frac{\left(1-\theta(1-\theta)\nu^2\psi^2(\xi_\ell)
					\right)\cos(\xi_\ell)-\nu\psi(\xi_\ell)\sin(\xi_\ell)}{1+\theta^2\nu^2\psi^2(\xi_\ell)} \nonumber \\
					&= \left(\frac{1}{m} \sum_{\ell=1}^{m} \frac{\left(1-\theta(1-\theta)\nu^2\psi^2(\xi_\ell)\right)
					\cos(\xi_\ell) - \nu\psi(\xi_\ell)\sin(\xi_\ell)}{1+\theta^2\nu^2\psi^2(\xi_\ell)}\right) -
					\frac{1}{m}\sum_{\ell=1}^m \cos(\xi_\ell) \nonumber \\
					&= -\frac{\nu}{m} \sum_{\ell=1}^{m} \frac{\psi(\xi_\ell)\sin(\xi_\ell)+\theta\nu\psi^2(\xi_\ell)
					\cos(\xi_\ell) }{1+\theta^2\nu^2\psi^2(\xi_\ell)} \nonumber \\
					&= -\frac{2\nu}{m} \sum_{\ell=2}^{m/2} \frac{\psi(\xi_\ell)(\sin(\xi_\ell)+\theta\nu\psi(\xi_\ell)
					\cos(\xi_\ell))}{1+\theta^2\nu^2\psi^2(\xi_\ell)},
	\end{align}
	where in the last equality we used $\psi(\xi_1) = \psi(0) = 0$,  $\psi(\xi_{m/2+1}) =  0$, and the symmetry
	of angles $\xi_\ell$ ($1 \le \ell \le m$) about the $x$-axis when $m$ is even (explicitly, the identities
	$\sin\left(k\frac{2\pi(\ell-1)}{m}\right)=-\sin\left(k\frac{2\pi(m-\ell+1)}{m}\right)$ and
	$\cos\left(k\frac{2\pi(\ell-1)}{m}\right)=\cos\left(k\frac{2\pi(m-\ell+1)}{m}\right)$ for positive integers
	$k$, $\ell$ and $m$).
	Define the function 
	\begin{align*}
		f(x; \theta, \nu) \coloneqq \frac{\psi(x)(\sin(x)+\theta\nu\psi(x)\cos(x))}{1+\theta^2\nu^2\psi^2(x)}.
	\end{align*}
	Then, we can express \eqref{M1m} by summing only over indices $\ell$ for which $0 < \xi_\ell  < \pi/2$ (and separating the case $\xi_\ell  = \pi/2$ when $m$ is divisible by 4),
	yielding
	\ifjournal
		\begin{align}\label{M1mnew}
			M_{1,m} = \begin{cases}
								-\dfrac{2\nu\psi(\pi/2)}{m\left(1+\theta^2\nu^2\psi^2(\pi/2)\right)} -
									\dfrac{2\nu}{m} \displaystyle{\sum\limits_{\ell=2}^{m/4}}
									\Big(f(\xi_\ell; \theta, \nu) + f(\pi-\xi_\ell; \theta, \nu) \Big),
									& \text{if } m \equiv 0 \Mod{4}, \\[20pt]
								-\dfrac{2\nu}{m} \displaystyle{\sum\limits_{\ell=2}^{(m+2)/4}}
									\Big(f(\xi_\ell; \theta, \nu) + f(\pi-\xi_\ell; \theta, \nu)\Big),
									& \text{if } m \equiv 2 \Mod{4}.
							\end{cases}
		\end{align}
	\else
		\begin{align}\label{M1mnew}
			M_{1,m} = \begin{cases}
								\dfrac{-2\nu\psi(\pi/2)}{m\left(1+\theta^2\nu^2\psi^2(\pi/2)\right)} -
									\dfrac{2\nu}{m} \displaystyle{\sum\limits_{\ell=2}^{m/4}}
									\Big(f(\xi_\ell; \theta, \nu) + f(\pi-\xi_\ell; \theta, \nu) \Big),
									& \text{if } m \equiv 0 \Mod{4}, \\[20pt]
								-\dfrac{2\nu}{m} \displaystyle{\sum\limits_{\ell=2}^{(m+2)/4}}
									\Big(f(\xi_\ell; \theta, \nu) + f(\pi-\xi_\ell; \theta, \nu)\Big),
									& \text{if } m \equiv 2 \Mod{4}.
							\end{cases}
		\end{align}
	\fi
	We will also use the identity
	\begin{align*}
		&f(\xi_\ell; \theta, \nu) + f(\pi-\xi_\ell; \theta, \nu) = \\
			&\frac{\Bigl(\psi(\xi_\ell) + \psi(\pi-\xi_\ell)\Bigr)
			\Bigl(\sin(\xi_\ell) + \theta\nu\psi(\xi_\ell)\cos(\xi_\ell) +
			\theta\nu\psi(\pi-\xi_\ell)\bigl(\theta\nu\psi(\xi_\ell)\sin(\xi_\ell)-\cos(\xi_\ell)\bigr)\Bigr)}
				{\bigl(1+\theta^2\nu^2\psi^2(\xi_\ell)\bigr)\bigl(1+\theta^2\nu^2\psi^2(\pi-\xi_\ell)\bigr)}.
	\end{align*}
	First, observe that $\sin(\xi_\ell)$ and $\cos(\xi_\ell)$ are positive for each index $2 \le \ell \le (m+2)/4$ in
	\eqref{M1mnew}. Now for the $4^\text{th}$-order centered spatial discretization, easy calculations show that $\psi(\pi/2) = 4/3$,
	
	\begin{align*}
		\psi(\xi_\ell) = \frac{1}{3}\sin(\xi_\ell)(4 - \cos(\xi_\ell)), \quad \text{and} \quad
		\psi(\pi - \xi_\ell) = \frac{1}{3}\sin(\xi_\ell)(4 + \cos(\xi_\ell)),
	\end{align*}
	so they are all positive as well. Let us fix $\theta \in (0,1]$ arbitrarily, and notice that for each $\ell$ we can find $\nu_\ell > 0$ such that
	$\theta\nu\psi(\xi_\ell)\sin(\xi_\ell)-\cos(\xi_\ell)>0$ for $\nu>\nu_\ell$. Let $\nu_0:=\max \nu_\ell$, then  
	$f(\xi_\ell; \theta, \nu) + f(\pi-\xi_\ell; \theta, \nu)>0$ for $\nu>\nu_0$. Therefore, $M_{1,m} < 0$ for any $\theta \in (0,1]$ and $\nu>\nu_0$.

\end{proof}
\begin{remark}\label{remark15} It is easily seen that the previous proof boils down to the fact that for the $4^\text{th}$-order centered spatial discretization we have $\psi(x)>0$ for $x\in(0,\pi)$, where
\[
\psi(x)=2 \left(\frac{2}{3}\sin (x)-\frac{1}{12} \sin (2 x)\right)=\frac{1}{3} \sin (x) \bigl(4-\cos (x)\bigr).
\]
We know from \eqref{psi} that for the $6^\text{th}$-order scheme
\[
\psi(x) = 2 \left(\frac{3}{4}\sin (x)-\frac{3}{20} \sin (2 x)+\frac{1}{60} \sin (3 x)\right) =
	\frac{1}{15} \sin (x) \bigl(23-9\cos (x)+\cos (2x)\bigr),
\]
so we again have $\psi(x)>0$ for $x\in(0,\pi)$. Therefore, the analogue of Proposition \ref{m_even_nonpositivityFE} is true for the $6^\text{th}$-order scheme as well (c.f.~the proof of Proposition \ref{m_even_nonpositivitySpectral}).
\end{remark}
\begin{remark}
	Based on the observations above Proposition~\ref{m_even_nonpositivityFE}, we conjecture the following:
	given an arbitrary finite difference centered discretization in space coupled with the $\theta$-method, then
	$M_{1,m} < 0$ for even $m$, and for \emph{all} values $\nu > 0$ and $\theta \in (0,1]$---that is, $\nu_0=0$ can be chosen in general.
\end{remark}
\item [{Case 2:} $m$ is {odd}.] In this case we have found that positivity preservation is possible for a suitable
set of $\theta\in(0,1]$ and $\nu>0$ values; see  Figures~\ref{fig_5stencilm579} and \ref{fig_7stencilm79}.

Also, if we assume $\psi(\xi_\ell) \ne 0$ for each $2\le\ell\le m$, then we can extend the asymptotic results of
Section~\ref{sectioncentered} for an arbitrary high-order centered discretization (c.f.~the ``odd $m$'' case in Section \ref{spectrcoll}).
\end{description}

\begin{figure}
\begin{center}
\includegraphics[width=0.48\textwidth]{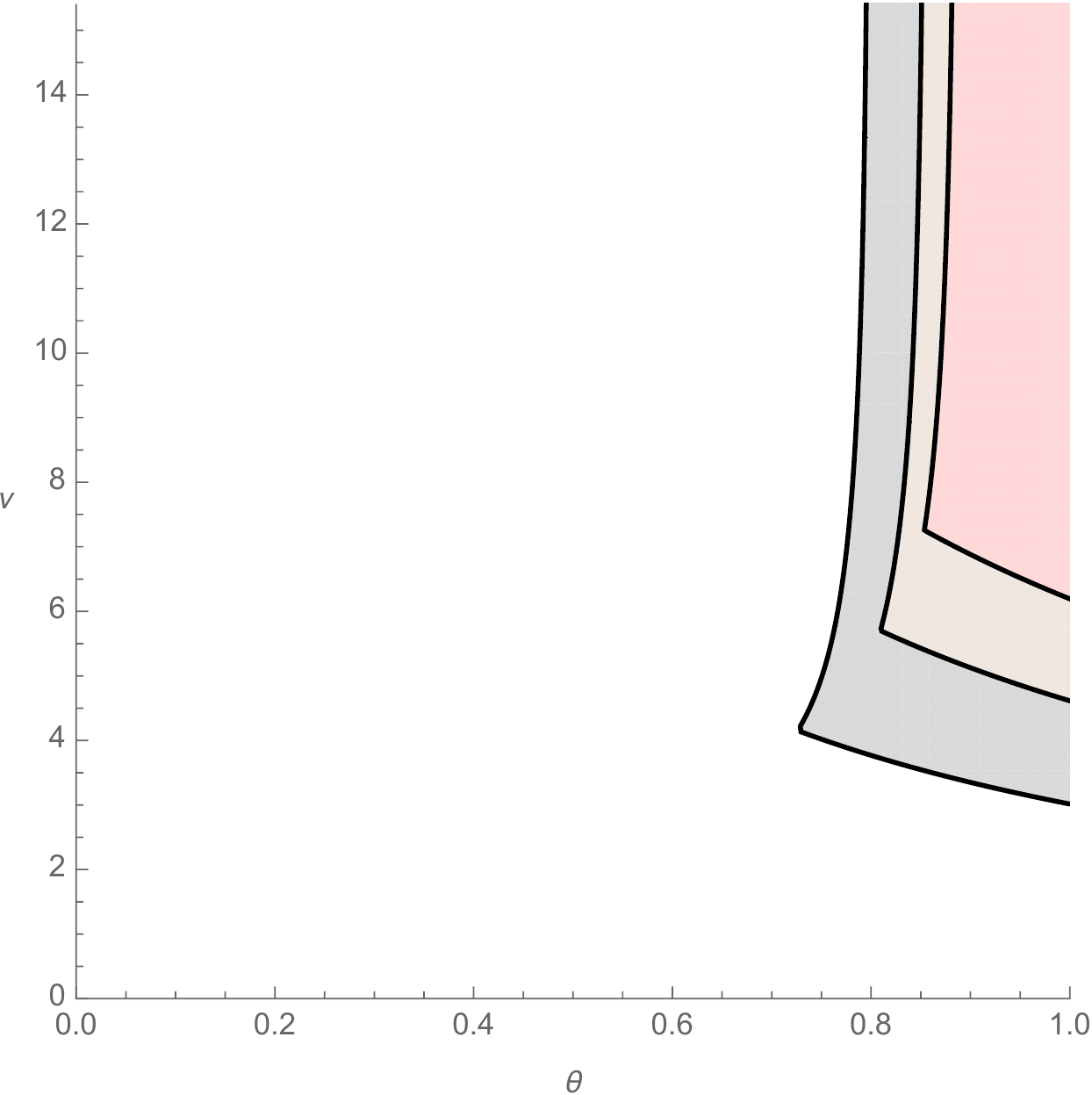}
\caption{Non-negativity of the matrix family generated by the $4^\text{th}$-order centered discretization in space and the $\theta$-method in time; see Section \ref{highercentered}. The parameter regions in the $(\te,\nu)$ parameter plane ensuring $M(m,\te,\nu)\ge 0$ are highlighted for $m\in\{5, 7, 9\}$ (different values of $m$ are represented by different colors; the regions ``shrink'' as $m$ gets larger). For $m=5$, $m=7$, and $m=9$, the ``lower left corner'' point of the region has $\theta\approx 0.726106$, $\theta\approx 0.809401$, and $\theta\approx 0.853562$, respectively.}\label{fig_5stencilm579}
\end{center}
\end{figure}

\begin{figure}
\begin{center}
\includegraphics[width=0.48\textwidth]{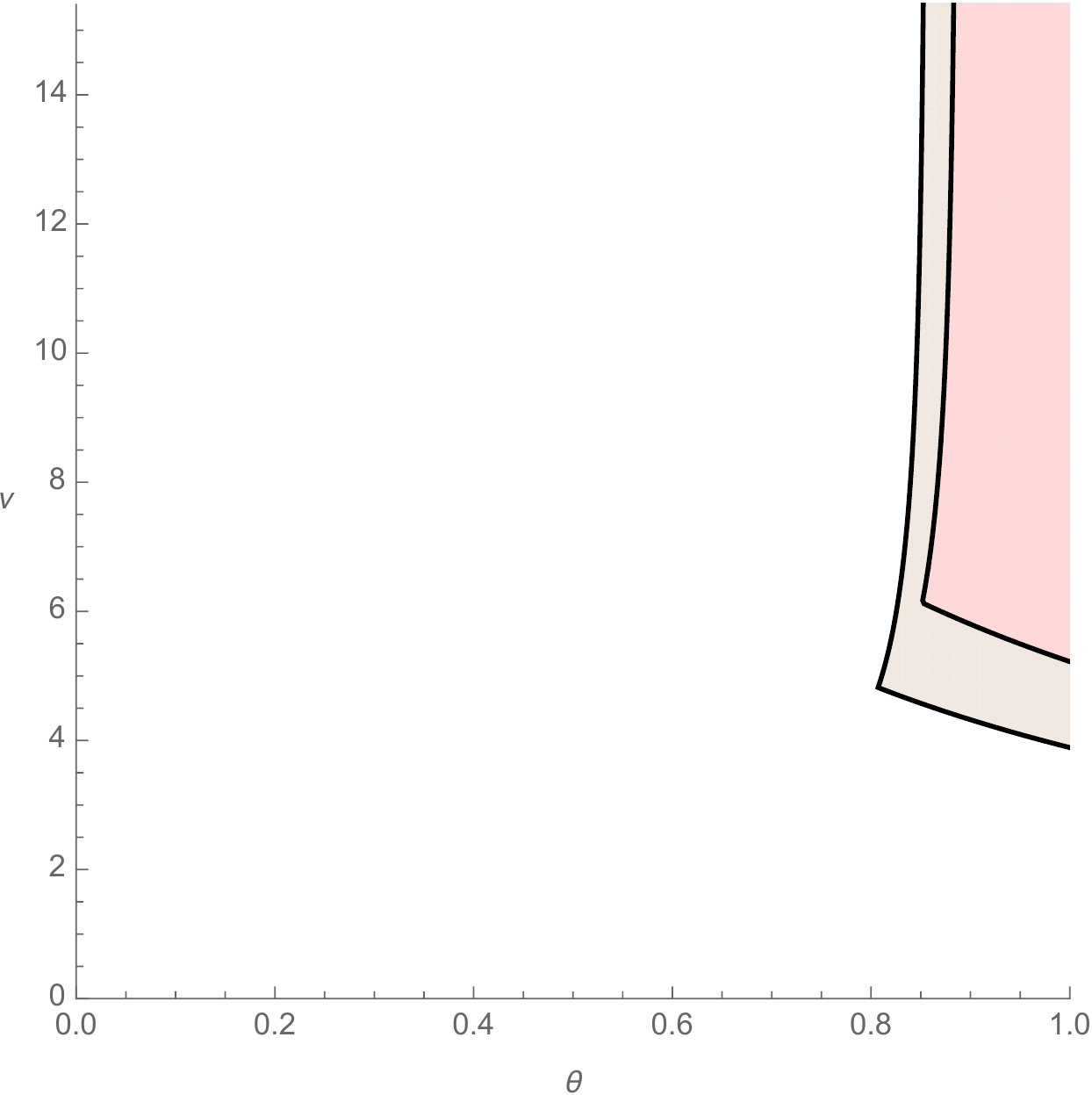}
\caption{Non-negativity of the matrix family generated by the $6^\text{th}$-order centered discretization in space and the $\theta$-method in time; see Section \ref{highercentered}. The parameter regions in the $(\te,\nu)$ parameter plane ensuring $M(m,\te,\nu)\ge 0$ are highlighted for $m\in\{7, 9\}$ (different values of $m$ are represented by different colors; the regions ``shrink'' as $m$ gets larger). For $m=7$ and $m=9$, the ``lower left corner'' point of the region has $\theta\approx 0.807042$ and $\theta\approx 0.851437$, respectively.}\label{fig_7stencilm79}
\end{center}
\end{figure}

\subsection{Fourier spectral collocation in space, \texorpdfstring{$\theta$}{}-method in time}\label{spectrcoll}
Here, we consider the spectral method that results from extending the
finite difference stencil to include the whole spatial grid.  In this section, we assume $m\ge 4$. The resulting spatial 
semi-discretization can again be written in the form \eqref{semi-discrete}
where the matrix $L$ takes the form \cite{trefethen,peyret}
\begin{align*}
	L_{i,j} = \begin{cases}
						0 & \text{if } i = j, \\[10pt]
						\dfrac{\pi}{m}(-1)^{i+j}\cot{\!\left(\frac{(i-j)\pi}{m}\right)} & \text{if } i \neq j,
					\end{cases}
\end{align*}
for even $m$, and
\begin{align*}
	L_{i,j} = \begin{cases}
						0 & \text{if } i = j, \\[10pt]
						\dfrac{\pi}{m}(-1)^{i+j}\csc{\!\left(\frac{(i-j)\pi}{m}\right)} & \text{if } i \neq j,
					\end{cases}
\end{align*}
for odd $m$.
As in Section~\ref{sectiondiscFourier}, the spectral collocation matrices have an eigendecomposition given by
\eqref{eigen}, with eigenvectors \eqref{eigenvectrors}, but now the entries of the diagonal matrix $\Lambda$ are
\begin{align}\label{lambdaSCdef}
    \imath\lambda_\ell & =
    \begin{cases}
        \imath\xi_\ell & 1 \le \ell < \frac{m}{2}+1, \\
        0 & \ell = \frac{m}{2}+1 \text{ and } m \text{ is even}, \\ 
        \imath(\xi_\ell - 2\pi) & \frac{m}{2}+1 < \ell \le m,
    \end{cases}
\end{align}
where $\xi_\ell$ has been defined in \eqref{xildef}.

When this matrix $L$ is coupled with the $\theta$-method as time discretization, the full discretization
matrix $M$ in \eqref{Mdualdef} is obtained.
The entries of the first row of the circulant matrix $M$ are again given by \eqref{M-entries}, and now
they take the form
\begin{align}\label{eq51}
	M_{1,j}  = \frac{1}{m} \sum_{\ell=1}^{m} \frac{\left(1-\theta(1-\theta)\nu^2\lambda_\ell^2\right)
  \cos((j-1)\xi_\ell) + \nu\lambda_\ell\sin((j-1)\xi_\ell)}{1+\theta^2\nu^2\lambda_\ell^2} \quad (1\le j \le m).
\end{align}
Similarly to the Sections~\ref{sectiondiscFourier} and \ref{highercentered}, we distinguish between even and
odd sizes of the discretization matrices $M$.
\begin{description}[style=unboxed,leftmargin=0cm]
\item [{Case 1:} $m$ is {even}.]
From \eqref{eq51} we have (also using \eqref{angle-identities}) that
\begin{align}\label{M1mSpectral}
	M_{1,m} = \frac{1}{m} \sum_{\ell=1}^{m} \frac{\left(1-\theta(1-\theta)\nu^2\lambda_\ell^2\right)
	\cos(\xi_\ell) - \nu\lambda_\ell\sin(\xi_\ell)}{1+\theta^2\nu^2\lambda_\ell^2}.
\end{align}
The following proposition proves that \eqref{M1mSpectral} is negative for sufficiently large CFL number $\nu$.
\begin{proposition}\label{m_even_nonpositivitySpectral}
	Consider the iterative formula \eqref{M} applied to the advection equation \eqref{advection} with periodic
	boundary condition.
	Let the matrix $M_{m \times m}$ result from a given even spactral collocation method in space (with $m$
	spatial grid points) and the $\theta$-method in time.
	Let also $\nu$ be the CFL number defined in \eqref{nudef}.
	If $m \ge 4$ is \emph{even}, then there exists $\nu_0  >0$ such that the matrix $M$ has at least one
	negative entry for any $\nu > \nu_0$.
\end{proposition}
\begin{proof}
	The proof is similar to that of Proposition~\ref{m_even_nonpositivityFE}: we show that
	$M_{1,m}<0$ for all $\theta \in (0,1]$ and $\nu > \nu_0$, where $\nu_0 >0$ depends on $m$ and $\theta$.
	
	First, observe that by using $\frac{1}{m}\sum_{\ell=1}^m \cos(\xi_\ell) = 0$ we can rewrite \eqref{M1mSpectral} as
	\[
	M_{1,m} = -\frac{\nu}{m} \sum_{\ell=1}^{m} \frac{\lambda_\ell\sin(\xi_\ell) +
		\theta\nu\lambda_\ell^2 \cos(\xi_\ell)}{1+\theta^2\nu^2\lambda_\ell^2}.
	\]
	Now, since $\lambda_1 = \lambda_{m/2+1} = 0$, we have
	\[
	M_{1,m} = -\frac{\nu}{m} \sum_{\ell=2}^{m/2} \left(\frac{\lambda_\ell\sin(\xi_\ell) +
		\theta\nu\lambda_\ell^2 \cos(\xi_\ell)}{1+\theta^2\nu^2\lambda_\ell^2}\right)-\frac{\nu}{m} \sum_{\ell=m/2+2}^{m} \left(\frac{\lambda_\ell\sin(\xi_\ell) +
		\theta\nu\lambda_\ell^2 \cos(\xi_\ell)}{1+\theta^2\nu^2\lambda_\ell^2}\right).
	\]
	One easily checks from \eqref{lambdaSCdef} and \eqref{xildef} that for any $\frac{m}{2}+2\le\ell\le m$ 
	\[
	\lambda _\ell=\xi _\ell-2 \pi =-\xi _{m+2-\ell}=-\lambda _{m+2-\ell}.
	\]
	This implies that 
	\begin{align*}
		\sum_{\ell=2}^{m/2} \frac{\lambda_\ell\sin(\xi_\ell) + \theta\nu\lambda_\ell^2 \cos(\xi_\ell)}
			{1+\theta^2\nu^2\lambda_\ell^2}
		= \sum_{\ell=m/2+2}^{m} \frac{\lambda_\ell\sin(\xi_\ell) + \theta\nu\lambda_\ell^2 \cos(\xi_\ell)}
			{1+\theta^2\nu^2\lambda_\ell^2},
	\end{align*}
	so 
	\[
		M_{1,m}  = -\frac{2\nu}{m} \sum_{\ell=2}^{m/2} \frac{\lambda_\ell\sin(\xi_\ell) +
			\theta\nu\lambda_\ell^2 \cos(\xi_\ell)}{1+\theta^2\nu^2\lambda_\ell^2}.
	\]
	Now notice (due to $\lambda_\ell=\xi_\ell$ for $2\le\ell\le \frac{m}{2}$) that we also have 
	\[
	M_{1,m}  = -\frac{2\nu}{m} \sum_{\ell=2}^{m/2} \frac{\xi_\ell\sin(\xi_\ell) +
			\theta\nu\xi_\ell^2 \cos(\xi_\ell)}{1+\theta^2\nu^2\xi_\ell^2},
	\]
	thus
	\[
	M_{1,m}  =-\frac{2\nu}{m} \sum_{\ell=2}^{m/2} \frac{\psi(\xi_\ell)(\sin(\xi_\ell)+\theta\nu\psi(\xi_\ell)
					\cos(\xi_\ell))}{1+\theta^2\nu^2\psi^2(\xi_\ell)}
	\]
	with $\psi(x):=x$. Since $\psi(x)>0$ for $x\in(0,\pi)$, according to Remark \ref{remark15}, the proof is complete.
\end{proof}

As previously, we conjecture that $M_{1,m}<0$ for \emph{all} values of $\nu>0$ and $\theta\in(0,1]$.
We have been able to verify this for $m\in\{4,6,8,10\}$, as follows.
The expression $M_{1,m}$ is a rational function in $\nu$ and $\theta$, whose
denominator is positive. By introducing a new variable $y:=\pi\theta\nu\ge 0$
and dividing by $\nu>0$, we can write the numerator as a univariate polynomial 
$p_m(y)$.  For example,
$$p_8(y) = -3 \left(8+3 \sqrt{2}\right) y^4+64 y^3-32 \left(7+5 \sqrt{2}\right) y^2+256 y-256 \left(2+\sqrt{2}\right).$$
We have confirmed with symbolic calculations that $p_m(y)<0$
for all $y\ge 0$ in the cases $m\in\{4,6,8,10\}$. 
\begin{remark} 
When generating the matrix $M$ for the symbolic calculations for larger values of $m$ for the actual full discretization, it is of course computationally more efficient to use the formula $\cF R(\nu \Lambda) \cF^*$ in \eqref{Mdualdef} instead of $R(\nu L)$ (because in the latter form one would need to evaluate the inverse of a non-sparse matrix).
\end{remark} 

\item [{Case 2:} $m$ is {odd}.]
This time we find a behavior similar to that observed in Section~\ref{sectioncentered}; see
Figure~\ref{fig_spectral}. \\
Moreover, since this time $\lambda_\ell\ne 0$ for $2\le\ell\le m$, the same asymptotic results hold as in the case of the
$2^\text{nd}$-order finite difference scheme in Section~\ref{sectioncentered}. Indeed, 
for odd $m = 2k+1 \ge 5$ we have that
\begin{align}\label{M1,j_m=odd_spectral}
	\begin{split}
		M_{1,1} & = \frac{1}{m} \left(1 + \sum_{\ell=2}^{m} \frac{1-\theta(1-\theta)\nu^2\lambda^2_\ell}
			{1+\theta^2\nu^2\lambda^2_\ell}\right) \\
		M_{1,j} & = \frac{1}{m} \left(1 + \sum_{\ell=2}^{m} \frac{(1-\theta(1-\theta)\nu^2\lambda^2_\ell)
			\cos((j-1)\xi_\ell) + \nu \lambda_\ell\sin((j-1)\xi_\ell)}{1+\theta^2\nu^2\lambda^2_\ell}
			\right) \quad (j\ge 2).
	\end{split}
\end{align}

Taking $\nu \to +\infty$ and $\te \in (0,1]$ fixed in \eqref{M1,j_m=odd_spectral}, yields
\begin{align*}
    M_{1,1}^\infty \coloneqq \lim_{\nu \to +\infty} M_{1,1} & =  1-\frac{m-1}{m\theta} \\
    M_{1,j}^\infty \coloneqq \lim_{\nu \to +\infty} M_{1,j} & = \frac{1}{m\theta} \quad (j \ge 2).
\end{align*}
As a result, 
we conclude that
\begin{align*}
	M_{1,j}^\infty > 0 \text{ for all } 2\le j\le m \text{ and } \te \in (0,1],
\end{align*}
while
\begin{align*}
	M_{1,1}^\infty > 0 \quad\Longleftrightarrow \quad \theta > \frac{m-1}{m}.
\end{align*}
\end{description}

\begin{figure}
\begin{center}
\includegraphics[width=0.48\textwidth]{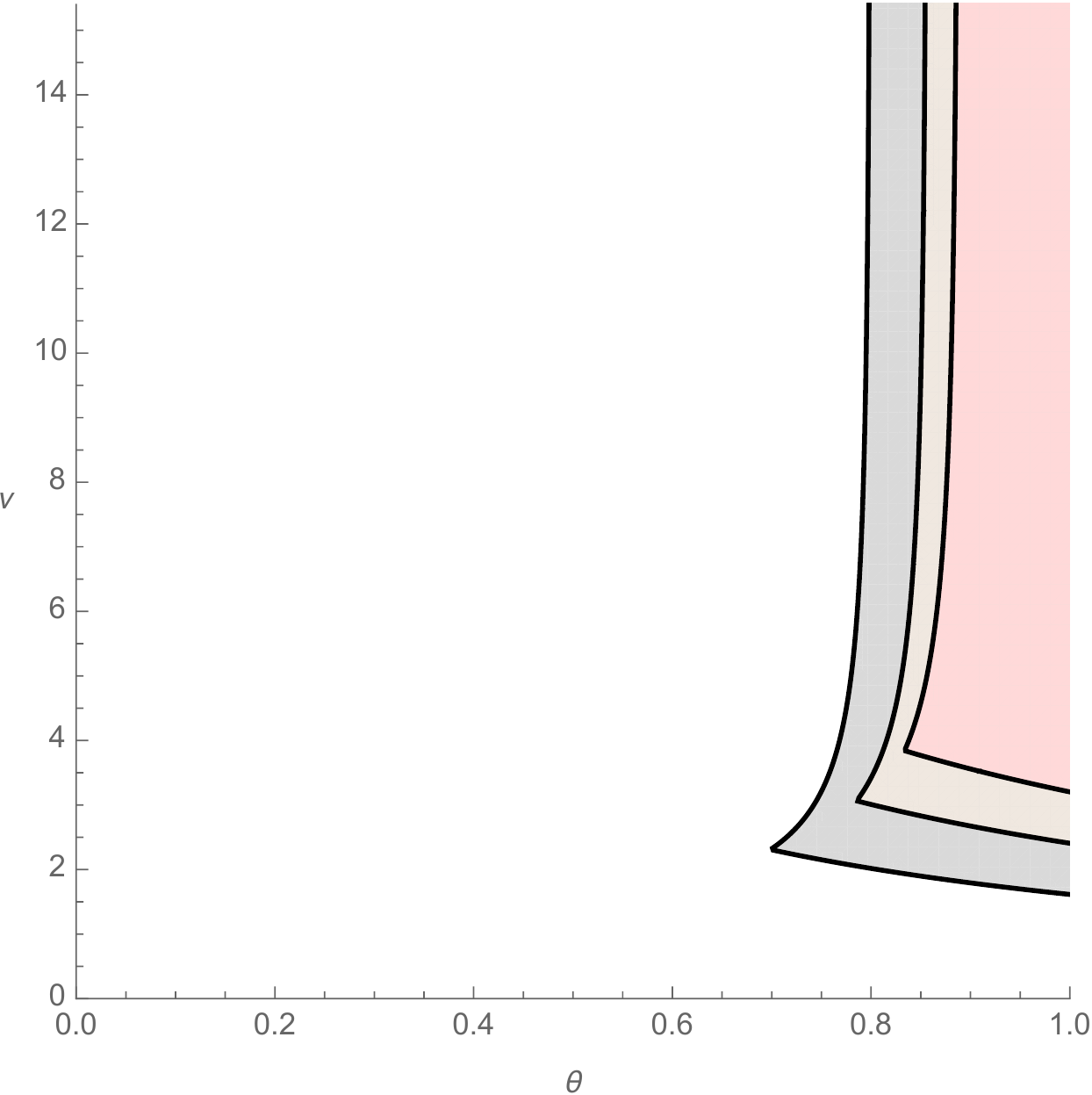}
\caption{Non-negativity of the matrix family generated by the Fourier spectral collocation method in space and the $\theta$-method in time; see Section~\ref{spectrcoll}. The parameter regions in the $(\te,\nu)$ parameter plane ensuring $M(m,\te,\nu)\ge 0$ are highlighted for $m\in\{5, 7, 9\}$ (different values of $m$ are represented by different colors; the regions ``shrink'' as $m$ gets larger).}\label{fig_spectral}
\end{center}
\end{figure}

\section{Conclusions}\label{conclusions}

In this work, we have studied the positivity preservation of certain fundamental
discretizations of the advection equation with periodic boundary condition \eqref{advection}. 
Our detailed investigations in Sections~\ref{sectiondiscFourier}--\ref{section3}
were devoted to the  full discretization obtained by coupling the second-order centered
differences in space with the $\theta$-method in time.  
 Rather than using SSP theory \cite{SSPbook}, 
we have employed a direct
approach, first based on discrete Fourier analysis and then
on a polynomial representation of the entries of the full discretization matrix.
The characterization of the 
matrix entries, 
along with the related trigonometric identities
presented in Corollary \ref{corollary1}, may be of independent interest.
In Section \ref{otherdisc}, we considered higher-order centered
differences or Fourier spectral collocation in space, and again the $\theta$-method in time. 

For all full discretizations constructed this way, we have found similar behavior.
If the number of spatial grid points $m$ is even, no method is positivity
preserving, while if $m$ is odd, some methods may be positivity
preserving.  Positivity is generally enhanced by taking larger values
of the CFL number $\nu>0$, larger values of the time-discretization parameter $\theta\in [0,1]$, or smaller
values of the spatial grid points $m\in\mathbb{N}^+$.  These tendencies, and more specific results, are described
in Theorem~\ref{thm2}, and can be seen
in Figures~\ref{fig_variousk}, \ref{fig_5stencilm579}, \ref{fig_7stencilm79},
and \ref{fig_spectral}.
Our positive results about the full discretizations are perhaps unexpected, since neither of the underlying spatial semi-discretizations preserves positivity.

Although some of the spatial discretizations considered above have high order, the $\theta$-method as time discretization typically has order only 1 (order 2 occurs only for $\theta=1/2$). Therefore, we emphasize that our goal in this work is not to provide efficient discretizations but rather
to understand the behavior of these simple building blocks, as a means
of gaining insight and understanding the positivity of more complicated
discretizations that may not be amenable to a thorough analysis.

There are several possible future directions for research building on this
work.  Other finite difference spatial discretizations could be studied
using similar techniques, and higher-order one-step time discretizations could
easily be incorporated via \eqref{Mdualdef}.  Similarly, finite difference
discretizations of other linear partial differential equations could be analyzed with the same techniques.
Further areas for extension might 
include other boundary conditions or
multidimensional problems.






\ifjournal
	\sloppy
	\printbibliography
\else
	\bibliographystyle{acm}
	\bibliography{references}
\fi

\end{document}